\documentclass[12pt, leqno]{article}
\usepackage{amsmath,amsfonts,amsthm,amscd,amssymb,verbatim}
\usepackage{color, ulem}
\usepackage[all]{xy}

\numberwithin{equation}{section}

\theoremstyle{plain}
\newtheorem{thm}{Theorem}[section]

\newtheorem{prop}[thm]{Proposition}
\newtheorem{lem}[thm]{Lemma}
\newtheorem{cor}[thm]{Corollary}
\newtheorem{conj}[thm]{Conjecture}

\newtheorem{question}[thm]{Question}

\theoremstyle{definition}
\newtheorem{rem}[thm]{Remark}

\newtheorem{claim}[thm]{Claim}

\newcommand{\N}{{\Bbb N}}
\newcommand{\Z}{{\Bbb Z}}
\newcommand{\Q}{{\Bbb Q}} 

\newcommand{\C}{{\Bbb C}}

\newcommand{\F}{{\Bbb F}}

\newcommand{\Spec}{{\mathrm{Spec}}\,}

\newcommand{\Spf}{{\mathrm{Spf}}\,}

\newcommand{\lra}{\longrightarrow}
\newcommand{\ra}{\rightarrow}
\newcommand{\hra}{\hookrightarrow}

\newcommand{\lla}{\longleftarrow}

\newcommand{\wt}[1]{\widetilde{#1}}

\newcommand{\ol}[1]{\overline{#1}}
\newcommand{\os}{\overset}

\newcommand{\dR}{{\mathrm{dR}}}

\newcommand{\crys}{{\mathrm{crys}}}

\newcommand{\conv}{{\mathrm{conv}}}

\newcommand{\Hom}{{\mathrm{Hom}}}

\newcommand{\Ker}{{\mathrm{Ker}}}
\newcommand{\im}{{\mathrm{Im}}}
\newcommand{\Coker}{{\mathrm{Coker}}}

\newcommand{\id}{{\mathrm{id}}}

\newcommand{\Coh}{{\mathrm{Coh}}}

\newcommand{\an}{{\mathrm{an}}}

\newcommand{\cD}{{\cal D}}
\newcommand{\cE}{{\cal E}}

\newcommand{\cM}{{\cal M}}
\newcommand{\cO}{{\cal O}}

\newcommand{\cS}{{\cal S}}

\newcommand{\cU}{{\cal U}}

\newcommand{\cX}{{\cal X}}

\newcommand{\wh}{\widehat}
\renewcommand{\wt}{\widetilde}

\newcommand{\et}{{\rm \acute{e}t}}

\newcommand{\MIC}{{\mathrm{MIC}}}
\newcommand{\qn}{{\mathrm{qn}}}
\newcommand{\HIG}{{\mathrm{HIG}}}

\newcommand{\Crys}{{\mathrm{Crys}}}
\newcommand{\Conv}{{\mathrm{Conv}}}

\begin{document}
\title{Convergent isocrystals on simply connected varieties}
\author{H\'el\`ene Esnault\footnote{
Freie Universit\"at Berlin, Arnimallee 3, 14195, Berlin, Germany; partly supported by the Einstein program and the ERC
Advanced
Grant 226257. } 
\,\,\, and \,\, 
Atsushi Shiho\footnote{
Graduate School of Mathematical Sciences, 
the University of Tokyo, 3-8-1 Komaba, Meguro-ku, Tokyo 153-8914, Japan;  partly supported by JSPS 
Grant-in-Aid for Scientific Research (C) 25400008, (B) 23340001 and (A)15H02048.} 
\setcounter{footnote}{-1}
}
\date{\today}
\maketitle

\begin{abstract}
It is conjectured by de Jong that, 
if $X$ is a connected smooth projective  variety over an algebraically closed 
field $k$ of characteristic $p>0$ with trivial \'etale fundamental group, 
any isocrystal
on $X$ is constant. We prove this conjecture under certain additional assumptions. 
\end{abstract}

\tableofcontents

\section*{Introduction}
The fundamental group is an important invariant in topology, algebraic geometry and arithmetic geometry. For a complex connected 
smooth projective variety $X$,
the topological fundamental group $\pi_1(X)$ (based at some point), which classifies all the coverings of $X$, is defined in a topological, non-algebraic way. But there are (at least) two approaches to define the fundamental group of $X$ in an algebraic way. 
One is the \'etale fundamental group $\pi_1^{\et}(X)$ \cite[V] {SGA1} (based at some geometric point), which classifies all the finite \'etale coverings of $X$. It is isomorphic to the profinite completion of $\pi_1(X)$. 
Another one is the category of $\cO_X$-coherent $\cD_X$-modules, which is equivalent to the category of finite dimensional complex linear representations of $\pi_1(X)$ via the Riemann-Hilbert correspondence. As for the relation between these two algebraic approaches, 
Mal\v{c}ev \cite{malcev} and Grothendieck  \cite{grothendieck} proved that, if $\pi_1^\et(X)=\{1\},$  then there are 
no non-constant $\cO_X$-coherent $\cD_X$-modules. 

\medskip 

This leads to the question of an analog for a connected smooth projective variety $X$ 
over an algebraically closed field $k$ of characteristic $p>0$, namely, the question to 
compare the \'etale fundamental group $\pi_1^{\et}(X)$ of $X$ and the category of  $\cO_X$-coherent $\cD$-modules on $X$. Due to the absence of the topological fundamental group 
of $X$, the relation between them is more mysterious. 

\medskip

One issue to precisely formulate the question above is that 
there are many versions of $\cD$-modules  which are defined on $X$. 
One can consider the full ring of differential operators, or 
the ring of PD-differential operators. One can consider with or without 
thickenings to the Witt ring $W$ of $k$, and one can impose 
various nilpotence or convergence conditions.\footnote{Also, 
we can consider $\cD$-modules on $X$ with Frobenius 
structure. In this case, a $p$-adic analogue of the Riemann-Hilbert 
correspondence, which gives an equivalence between the category of $p$-adic representations of $\pi_1^{\et}(X)$ and the category of unit-root convergent $F$-isocrystals on $X$,
is known by Crew \cite{crew}.  Although it is also interesting to consider the 
case with Frobenius structure, we concentrate to the case without Frobenius structure 
in this article.
}
\medskip 

When we consider the full ring of differential operators $\cD_X$ on $X$ 
in usual sense (without any thickenings to $W$), 
the category of $\cO_X$-coherent 
$\cD_X$-modules is equivalent to the category 
${\rm Inf}(X/k)$ of crystals of finite 
presentation on the infinitesimal site $(X/k)_{\rm inf}$ of $X$ over $k$. 
In this case, Gieseker \cite{Gie75} conjectured in 1975 
that, on a connected smooth projective variety $X$ over an algebraically closed 
field $k$ of characteristic $p>0$ with $\pi_1^\et(X)=\{1\}$, there are no non-constant 
$\cO_X$-coherent $\cD_X$-modules. 
This conjecture was answered affirmatively in \cite[Thm.~1.1]{esnaultmehta}.

\medskip 

When we consider the full ring of differential operators on $X$ 
with thickenings to $W$, we obtain  
the category ${\rm Inf}(X/W)$ of crystals of finite 
presentation on the infinitesimal site $(X/W)_{\rm inf}$ of $X$ over $W$ 
(see Section \ref{sec:prel} for the definition).  
This is a $W$-linear category which lifts ${\rm Inf}(X/k)$. 
Because this category contains $p$-power torsion objects which cannot be 
constant even when $\pi_1^{\et}(X) = \{1\}$, it is natural to consider 
the $\Q$-linearization ${\rm Inf}(X/W)_{\Q}$, which we call the category 
of infinitesimal isocrystals. 
This category is known to be too small to capture all the geometric objects, 
but it is still an interesting category because it contains the geometric 
objects coming from finite \'etale coverings of $X$. 

\medskip 

The category of (certain) modules on the ring of PD-differential operators 
on $X$ with thickenings to $W$ and quasi-nilpotence condition is studied
most extensively, which is defined as    
the category ${\rm Crys}(X/W)$ of crystals of finite 
presentation on the crystalline site $(X/W)_{\rm crys}$ of $X$ over $W$ 
(see Section \ref{sec:prel} for the definition). As before, it is natural to consider 
the $\Q$-linearization ${\rm Crys}(X/W)_{\Q}$, which we call the category 
of isocrystals. 

\medskip 

After \cite{esnaultmehta}, de Jong conjectured in 2010 that, on a connected smooth projective variety $X$ over an algebraically closed field $k$ of characteristic $p>0$ with $\pi_1^\et(X)=\{1\}$, there are no non-constant isocrystals on $X$. 

\medskip

In this article, we consider the conjecture of de Jong for 
a closely related and slightly smaller category 
$\Conv(X \allowbreak /K)$ (where $K$ is the fraction field of $W$), the category 
of convergent isocrystals on $X$ over $K$, which is introduced by 
Ogus \cite[Defn.~2.7]{ogus1}. This corresponds to 
the category of (certain) modules on the ring of differential operators 
on $X$ with thickenings to $W$, tensorization with $\Q$ and the convergence condition, which is slightly 
stronger than the quasi-nilpotence condition. 
Although the category $\Conv(X/K)$ is slightly smaller than 
${\rm Crys}(X/W)_{\Q}$, it is large enough to contain 
the objects coming from geometry 
(e.g., the Gau\ss-Manin convergent isocrystals 
defined by Ogus \cite[Thm.~3.7]{ogus1})
and enjoys nice topological properties 
such as proper descent. Over an algebraically closed field,  
it shares many properties with  the category of lisse $\bar{ \mathbb{Q}}_\ell$-sheaves.

\medskip 

The conjecture of de Jong (for convergent isocrystals) 
is not trivial even when 
$X$ is liftable to a smooth projective scheme 
$X_W$ over $\Spec W$, because the \'etale fundamental 
group of the geometric generic fiber of $X_W$ need not be 
trivial. On the other hand, for 
a proper smooth morphism $f: Y \lra X$ which is 
liftable to a proper smooth morphism 
$f_W: Y_W \lra X_W$, we can prove 
the constancy of the Gau\ss-Manin convergent isocrystal 
$R^if_{\conv *} \cO_{Y/K}$ \cite[Thm.~3.7]{ogus1} 
rather easily, in the following way.
If we denote by $f_L: Y_L \lra X_L$ the base change of 
$f_W$ to a field $L$ containing $W$,  it suffices to prove 
the constancy of $R^if_{K, \dR *}\cO_{Y_K}$ 
as a module with an integrable connection, 
by \cite[Thm.~3.10]{ogus1}. 
Then we may assume that 
$K \subseteq \C$ and it suffices to prove the constancy of 
$R^if^{\an}_{\C,*}\C_{Y^{\an}_\C}$, where 
$f^{\an}_{\C}: Y^{\an}_{\C} \lra X^{\an}_{\C}$ is the 
analytification of $f_{\C}$. 
This is reduced to the constancy of  
$R^if_{\C,*}\Q_{\ell}$ for a prime $\ell \not= p$
by Artin's comparison theorem 
\cite[XVI,4]{SGA4}, and reduces to 
the constancy of $R^if_{W,*}\Q_{\ell}$, which 
is true by 
Grothendieck's base change  theorem on the \'etale fundamental group $\{1\} = \pi^{\et}_1(X) \xrightarrow{\cong } \pi^{\et}_1(X_W)$ \cite[Thm.~18.1.2]{EGAIV4}. 
Thus we find this conjecture interesting enough. 

\medskip 

Our main result is a partial solution to the conjecture of de Jong for 
convergent isocrystals, which is stated as follows. 

\begin{thm} \label{thmintro1}
 Let $X$ be a connected smooth projective  variety over $k$ with 
trivial \'etale fundamental group. Then
\begin{itemize}
\itemsep0em 
\item[(1)] any convergent isocrystal which is filtered so that the associated graded is a sum of rank $1$   
convergent isocrystals, is constant;
\item[(2)] if  the maximal slope of the sheaf of $1$-forms on $X$ is non-positive, 
then any convergent isocrystal is constant.
\end{itemize}
\end{thm}
We refer to Theorem~\ref{thm2} (1) and Corollary~\ref{rk1negativeslope} for this formulation. 
Theorem~\ref{thm2} is more precisely formulated: 
In (2), positive slopes of  the sheaf of $1$-forms on $X$ are allowed, 
according to the maximal rank of the irreducible constituents of 
the given convergent isocrystal. 

\medskip

We now explain the main ideas of the proof.  Convergent isocrystals are known to be Frobenius divisible, although $p$-torsion free crystals in one isocrystal class 
(called {\it lattices} of an isocrystal) is not. Using this, one proves in Proposition~\ref{prop:chern} that the Chern classes of the value $E_X$ on $X$ of a crystal $E$ over $W$ vanish when $E$ is a lattice of a convergent isocrystal. 

\medskip 

If one assumes in addition that $E_X$ is strongly $\mu$-semistable, one sees that 
the subquotients associated to some filtration of $E_X$ 
yield points in the moduli of $\chi$-stable sheaves with vanishing Chern classes. Then, assuming now that $X$ has trivial fundamental group, it is proved in  Propositions~\ref{prop:a}, \ref{prop:b}, \ref{prop:c} 
by a noetherianity argument,  that not only infinitely Frobenius divisible sheaves are constant (Gieseker's conjecture proved in \cite{esnaultmehta}), but also strongly $\mu$-semistable ones with vanishing Chern classes which admit a large enough Frobenius divisibility. 

\medskip 

This, together with the crystalline deformation theory 
in Propositions~\ref{prop:e}, \ref{prop:d}, Corollary~\ref{cor:d}
which allows to prove the constancy modulo $p^n$ from that 
on $X$, leads to the theorem (see Theorem~\ref{thm1}).
\begin{thm} \label{thmintro2}
Let $X$ be a connected smooth projective variety over $k$ with
trivial  \'etale fundamental group and let $\cE$ be a convergent 
isocrystal. If, for any $n \in \N$, 
the $F^n$-division $\cE^{(n)}$ of $\cE$ 
admits a lattice such that 
its value on $X$ as a coherent $\cO_X$-module is
strongly $\mu$-semistable, then $\cE$ is trivial. 
\end{thm}

Also, 
one proves a Langton type theorem in Proposition~\ref{prop:langton}, claiming 
the existence of a lattice whose restriction to a crystal on $X$ over $k$ is 
$\mu$-semistable. One proves 
Theorem~\ref{thmintro1}, together with its refinements not discussed in the introduction,  using Theorem~\ref{thmintro2} and the slope condition on the sheaf of $1$-forms which forces the requested stability conditions (see Proposition~\ref{prop:mr}).

\medskip 

As another consequence of Theorem~\ref{thmintro2}, 
we have the following corollary, which confirms the conjecture of de Jong 
for infinitesimal isocrystals (see Corollary~\ref{cor}). 

\begin{cor} \label{corintro}
Let $X$ be a connected smooth projective variety over $k$ with 
trivial \'etale fundamental group. Then any infinitesimal isocrystal on 
$X$ is constant. 
\end{cor}

\medskip 

{\it Acknowledgement:} The first named author thanks  Johan de Jong for  inviting her in the fall of 2010 to lecture  on \cite{esnaultmehta} at Columbia University, after which he formulated the conjecture discussed here. She thanks
Arthur Ogus for subsequent fruitful discussions which solved
  the abelian case (see Proposition~\ref{rem:eo} for a precise statement), Yves Andr\'e for discussions on the  moduli points of  isocrystals in Simpson's  moduli of flat connections on a characteristic $0$ lift of $X$, Adrian Langer for enlightening discussions on  \cite[Thm.~5.1]{langerlie}. She thanks
 Takeshi Saito for offering hospitality in summer 2014  at the Graduate School of Mathematical Sciences, the University of Tokyo,  where this work was started (see \cite{Shi14}).  The second named author thanks
 Kirti Joshi for his suggestion to consider 
the condition $\mu_{\max}(\Omega^1_X) \leq 0$. Both authors thank one of the anonymous referees for gracious, constructive  and helpful comments.

\section{Preliminaries} \label{sec:prel}

In this section, we review some facts on (iso)crystals, infinitesimal (iso)crystals, convergent isocrystals, and Cartier transform of 
Ogus-Vologodsky.

\medskip

Throughout the article,  we fix
 an algebraically closed field $k$ of characteristic $p>0$. Let 
$W$ be the Witt ring of $k$ and 
$K$  be the fraction field of $W$. For $n \in \N$, put 
$W_n := W/p^nW$. Let $\sigma: k \lra k$ be the Frobenius map 
$a \mapsto a^p$ on $k$. Let $\sigma_W: W \lra W$ be the 
automorphism which lifts $\sigma$ and let $\sigma_K: K \lra K$ be 
the automorphism induced by $\sigma_W$. 

\medskip

First we summarize a few facts on (iso)crystals from \cite[Sections~5/6/7]{berthelotogus}, \cite[III/IV]{berthelotbook}.
For a scheme $X$ of finite type over $k$, let 
$(X/W)_{\crys}$ (resp. $(X/W_n)_{\crys}$) 
be the {\it crystalline site} 
on $X/W$ (resp. $X/W_n$). An object 
is a pair $(U \hra T, \delta)$, where 
$U \hra T$ is a closed immersion over $W_n$ for some $n$ 
(resp. over $W_n$) from an open subscheme $U$ of $X$ to a scheme $T$ and 
$\delta$ is 
a PD-structure on $\Ker(\cO_T \ra \cO_U)$, compatible 
with the canonical PD-structure on $pW_n$. Morphisms are the obvious ones. 
For the definition of coverings, see \cite[page~5.2]{berthelotogus}. 
The structure sheaf $\cO_{X/W}$ on $(X/W)_{\crys}$ (resp.   $\cO_{X/W_n}$ on $(X/W_n)_{\crys}$) 
is defined by the rule $(U \hra T, \delta) \mapsto \Gamma(T,\cO_T)$. 

\medskip

A sheaf $E$ of $\cO_{X/W}$-modules (resp. $\cO_{X/W_n}$-modules) on 
 $(X/W)_{\crys}$ (resp. on $(X/W_n)_{\crys}$) is equivalent to the datum of a sheaf of 
 $\cO_T$-modules $E_T$ in the Zariski topology of $T$ for each object $T := (U \hra T, \delta)$, and of
an $\cO_{T'}$-linear morphism $\varphi^*E_{T} \ra E_{T'}$ for each 
morphism $\varphi: (U' \hra T', \delta') \ra (U \hra T, \delta)$, which is an isomorphism when $\varphi:T' \ra T$ is an open immersion and 
$U' $ is equal to $U \times_T T'$. The sheaf $E_T$ is called the {\it value} of $E$ at $T$. 
Via the module structure of $\cO_T$ over itself, $\cO_{X/W}$ (resp. $\cO_{X/W_n}$) is a sheaf of $\cO_{X/W}$-modules (resp. $\cO_{X/W_n}$-modules).
 
 \medskip
 
A sheaf $E$ of $\cO_{X/W}$-modules (resp. of $\cO_{X/W_n}$-modules) on 
$(X/W)_{\crys}$ (resp. on $(X/W_n)_{\crys}$) is a {\it crystal} if 
the morphisms $\varphi^*E_{T} \ra E_{T'}$ are all isomorphisms.  A crystal is 
 {\it  of finite presentation}   if its value $E_T$ is   an $\cO_T$-module of 
finite presentation    for any $(U \hra T, \delta)$. 
The {\it category of  crystals of finite presentation}  on 
$(X/W)_{\crys}$ (resp. on $(X/W_n)_{\crys}$)  is denoted by 
$\Crys(X/W)$ (resp. $\Crys(X/W_n)$), as a full subcategory of the category of sheaves of  $\cO_{X/W}$-modules (resp.  of $\cO_{X/W_n}$-modules). 
The structure sheaf $\cO_{X/W}$ (resp. $\cO_{X/W_n}$) is a crystal. 
It is known \cite[IV Prop. 1.7.6]{berthelotbook}
that $\Crys(X/W)$, $\Crys(X/W_n)$ are abelian categories. 
Furthermore, the categories 
$\Crys(X/W)$, $\Crys(X/W_n)$ satisfy the descent property 
for Zariski coverings of $X$, that is, crystals `glue' in the Zariski 
topology. 

\medskip 

 If we denote the topos associated to 
$(X/W)_{\crys}$ by $(X/W)_{\crys}^{\sim}$, it is 
functorial with respect to $X/W$, namely, if we have a commutative diagram 
\begin{equation}\label{eq:funct}
\begin{CD}
X' @>{f}>> X \\ 
@VVV @VVV \\ 
\Spf (W) @>{f'}>> \Spf (W)
\end{CD}
\end{equation}
with $X'$ of finite type over $k$, we have the canonical morphism of topoi 
$(X'/W)_{\crys}^{\sim} \lra (X/W)_{\crys}^{\sim}$ \cite[\S 5]{berthelotogus}. 
It induces the morphism of 
ringed topoi $((X'/W)_{\crys}^{\sim}, \cO_{X'/W}) \allowbreak \lra 
((X/W)_{\crys}^{\sim},\cO_{X/W})$ and the pullback functor 
$f^*: \Crys(X/W) \allowbreak \lra \Crys(X'/W)$. 
Similar functoriality holds also for 
the the ringed topos $((X/W_n)_{\crys}^{\sim}, \allowbreak \cO_{X/W_n})$ 
associated to $(X/W_n)_{\crys}$ and the category 
$\Crys(X/W_n)$. 

\medskip

The natural inclusion of sites 
$(X/W_n)_{\crys} \hra (X/W)_{\crys}$  induces 
the restriction functor 
\begin{equation}\label{rest}
\Crys(X/W) \lra \Crys(X/W_n); \,\,\, E \mapsto E_n. 
\end{equation}
Since $(X/W)_{\crys}$ is the $2$-inductive limit of the sites
$(X/W_n)_{\crys}$ \cite[page~7-22]{berthelotogus}, 
we have the equivalence 
\begin{equation}\label{proj}
\Crys(X/W) \os{\simeq}{\lra} \varprojlim_n \Crys(X/W_n); 
\,\,\, E \mapsto (E_n)_n. 
\end{equation}
The functors \eqref{rest}, \eqref{proj} are also functorial with respect to 
$X/W$.

\medskip 

For any object $ (U \hra T, \delta)$ in 
$(X/W)_{\crys}$ (resp. $(X/W_n)_{\crys}$), 
the {\it  functor of evaluation} at $T$ 
\begin{equation}\label{evt}
\Crys(X/W) \lra \Coh(\cO_T) \,\,\,\text{(resp. 
$\Crys(X/W_n) \lra \Coh(\cO_T)$)}; \,\,\, E \mapsto E_T 
\end{equation}
is defined, where $\Coh(\cO_T)$ denotes the  category of  sheaves of $\cO_T$-modules of finite type. It is known to be right exact. This follows from 
\cite[IV Rem.~1.7.8]{berthelotbook} and \cite[III Prop.~1.1.5]{berthelotbook}. 

\medskip

If $U \hra Y$ is a closed immersion from an open 
subscheme $U$ of $X$ into a smooth scheme $Y$ over $W_n$ and 
$T:= (U \hra T, \delta)$ is its PD-envelope \cite[page 3.20]{berthelotogus}, the functor 
\eqref{evt} is exact \cite[IV Prop. 1.7.5]{berthelotbook}. 
Moreover, we have an equivalence of categories lifting  
\eqref{evt}, which we explain now. The derivation 
$d: \cO_Y \ra \Omega^1_{Y}$ is extended canonically to 
a PD-derivation 
$d: \cO_{T} \ra \Omega^1_{T} := \cO_{T} \otimes_{\cO_Y} 
\Omega^1_{Y}$, and we have the notion of 
 $\cO_T$-{\it modules  of finite presentation 
with integrable connection on $T$ with respect to 
this PD-derivation} \cite[Section~4]{berthelotogus}. With the obvious morphisms, we denote the category of such objects by 
$\MIC(T)$, and denote the full subcategory consisting of 
quasi-nilpotent ones by $\MIC(T)^{\qn}$. 
(For the definition of the quasi-nilpotence, see 
\cite[Def.~4.10]{berthelotogus}.) 
For $T= (U \hra T, \delta)$ as above and $E \in \Crys(U/W_n)$, 
$E_T$ is naturally endowed with a quasi-nilpotent integrable connection 
$\nabla_{E_T} : E_{T} \ra E_{T} \otimes \Omega^1_{T}$, and 
we have a natural equivalence of abelian categories 
\begin{equation}\label{evt2}
\Crys(U/W_n) \os{\simeq}{\lra} 
\MIC(T)^{\qn}; \,\,\, E \mapsto (E_T, \nabla_{E_T}).  
\end{equation}

\medskip 

We use the functors \eqref{evt}, \eqref{evt2} in the following cases. 
 First, for a smooth variety $X$ over $k$, 
we have the right exact functors 
\begin{align}
& \Crys(X/W) \lra \Coh(\cO_X), \label{evxw} \\ 
& \Crys(X/W_n) \lra \Coh(\cO_X) \label{evxwn} 
\end{align} 
of evaluation at $X := (X \os{\id}{\hra} X, 0)$.
When $n=1$, the functor \eqref{evxwn} is exact. 
The functor \eqref{evxw} (resp. \eqref{evxwn}) 
is functorial with respect to $X/W$ (resp. $X/W_n$).

\medskip

Next, let $X$ be a smooth variety over $k$ and 
assume that we have a lifting of $X$ to a $p$-adic 
smooth formal scheme $X_W$ over $W$. As $X/k$ is smooth,  there always
exists such a lifting on affine open subschemes of $X$.  If we put 
$X_n := X_W \otimes_W W_n$, the evaluation 
at $X_n := (X \hra X_n, \text{canonical PD-structure on $p\cO_{X_n}$})$ 
induces the equivalence 
\begin{equation}\label{evpn}
\Crys(X/W_n) \os{\simeq}{\lra} \MIC(X_n)^{\qn}. 
\end{equation}
So we have an equivalence of categories 
\begin{align}
& \Crys(X/W) \os{\simeq}{\lra} \varprojlim_n \MIC(X_n)^{\qn} 
=: \MIC(X_W)^{\qn}; \label{evformal} \\ 
& E \mapsto \varprojlim_n (E_{X_n}, \nabla_{E_{X_n}}) =: 
(E_{X_W}, \nabla_{E_{X_W}}). \nonumber 
\end{align}
In addition, there exists a full embedding 
\begin{equation}\label{mic-d}
\MIC(X_W)^{\qn} \hra \text{($\cO_{X_W}$-coherent left $\wh{\cD}_{X_W/W}^{(0)}$-modules)}, 
\end{equation} 
where $\wh{\cD}_{X_W/W}^{(0)} := \varprojlim \cD_{X_n/W_n}^{(0)}$ and 
$\cD_{X_n/W_n}^{(0)}$ is the ring of PD-differential operators, by \cite[Thm.~4.8]{berthelotogus}. 
When there exists a local basis $x_1, ..., x_d$ of $X_W$ over $W$, 
$\wh{\cD}_{X_W/W}^{(0)}$ is topologically generated by 
$$\left(\dfrac{\partial}{\partial x}\right)^n := 
\left(\dfrac{\partial}{\partial x_1}\right)^{n_1} \left(\dfrac{\partial}{\partial x_2}\right)^{n_2} 
\cdots \left(\dfrac{\partial}{\partial x_d}\right)^{n_d} \quad (n := (n_1, ..., n_d) \in \N^d) $$
over $\cO_{X_W}$ \cite[Section 4]{berthelotogus}.

\medskip 

The functors \eqref{evformal}, \eqref{mic-d} are functorial with respect to $X_W/W$, 
namely, if there exists a diagram 
\begin{equation}\label{eq:functw}
\begin{CD}
X'_W @>{f_W}>> X_W \\ 
@VVV @VVV \\ 
\Spf (W) @>{f'}>> \Spf (W)
\end{CD}
\end{equation}
lifting \eqref{eq:funct}, the functor \eqref{evformal} is compatible with 
the pullback $f^*$ on the left hand side and the pullback $f_W^*$ 
on the right hand side. Also, the functor \eqref{mic-d} is compatible with 
the pullback $f_W^*$.

\medskip

We say that an object $(E_{X_W},\nabla_{E_{X_W}})$ (resp. $(E_{X_n},\nabla_{E_{X_n}})$) in $\MIC(X_W)^{\qn}$ (resp. in
$\MIC(X_n)^{\qn}$) is $p$-torsion if so is $E_{X_W}$ (resp.$ E_{X_n}$). 
Since the restriction functors
\begin{align*}
& \text{($p$-torsion objects in $\MIC(X_W)^{\qn}$)} \lra \MIC(X)^{\qn}, \\ 
& \text{($p$-torsion objects in $\MIC(X_n)^{\qn}$)} \lra \MIC(X)^{\qn} 
\end{align*}
are equivalences, we see by \eqref{evpn}, \eqref{evformal} and Zariski descent 
that the restriction functors 
\begin{align*}
& \text{($p$-torsion objects in $\Crys(X/W)$)} \lra \Crys(X/k), \\ 
& \text{($p$-torsion objects in $\Crys(X/W_n)$)} \lra \Crys(X/k)
\end{align*} 
are equivalences. 

\medskip

When $X$ is projective over $k$ and we are given a 
fixed closed $k$-immersion $\iota: X \hra \mathbb{P}_k^N$, 
we denote the PD-envelope of $X \os{\iota}{\hra} \mathbb{P}_k^N \hra 
\mathbb{P}_{W_n}^N$ by $D_n := (X \hra D_n, \delta_n)$. 
Then the equivalence of categories \eqref{evt2} becomes globally defined on $X$ \cite[Thm. 6.6]{berthelotogus}
\begin{equation}\label{evdn}
\Crys(X/W_n) \os{\simeq}{\lra} \MIC(D_n)^{\qn}; \,\,\, E \mapsto (E_{D_n}, \nabla_{E_{D_n}}). 
\end{equation}

\medskip

A crystal $E \in  \Crys (X/W)$ (resp. $\Crys(X/W_n)$) is called 
{\it locally free} if, for any object $(U\hookrightarrow T, \delta)$ in $(X/W)_{\crys}$ 
(resp. $\Crys(X/W_n)$), $E_T$ is locally free of finite rank. 

\medskip
A crystal $E \in \Crys(X/W)$ is  said to be $p$-{\it torsion free} if the multiplication by $p$ on $E$ is injective. 

\medskip

A crystal $E \in \Crys(X/W_n)$ is 
called  {\it flat} over $W_n$ if, for any $1 \leq i \leq n-1$, 
the morphism $E/p^{n-i}E \lra E$ induced by 
the multiplication by $p^i$ is an isomorphism onto its image $p^iE\subset E$. 

\medskip

When we have a lifting of $X$ to a $p$-adic smooth 
 formal scheme $X_W$ over $W$, $E \in  \Crys (X/W)$ is locally free 
(resp. $p$-torsion free) if and only if $E_{X_W}$ is 
locally free (resp. $p$-torsion free), and 
$E \in \Crys(X/W_n)$ is flat over $W_n$ if and only if $E_{X_n}$ is 
flat over $W_n$, where $X_n := X_W \otimes_W W_n$. 
Therefore, 
$E \in \Crys(X/W)$ is locally free if and only if  its value 
$E_X\in \Coh(X)$ is locally free, if and only if its 
restriction  to $\Crys(X/k)$ is locally free. 
Also, 
$E \in \Crys(X/W)$ is $p$-torsion free if and only if 
$E_n$ is flat over $W_n$ for each $n \in \N$, where 
$(E_n)_n \in \varprojlim_n \Crys(X/W_n)$ is the object 
corresponding to $E$ via \eqref{proj}. 

\medskip

For a smooth variety $X$ over $k$, 
let $\Crys(X/W)_{\Q}$ be the $\Q$-linearization of 
the category $\Crys(X/W)$, which is called the category of 
{\it isocrystals} on $X$. This means that the objects 
of  $\Crys(X/W)_{\Q}$ 
are those of  $\Crys(X/W)$ and that the morphisms of  $\Crys(X/W)_{\Q}$ are those of 
 $\Crys(X/W)$ tensored with $\Q$. So one has a natural functor 
 $\Crys(X/W) \xrightarrow{\Q \otimes}  \Crys(X/W)_{\Q}$ which is the identity on objects. 
The image of $\cE$ by this functor is denoted by  $\Q \otimes \cE$. 
When $X$ is liftable to a $p$-adic smooth formal scheme $X_W$ over $W$, 
the functors \eqref{evformal} and \eqref{mic-d} induce the full embedding 
\begin{equation}\label{mic-dq}
\Crys(X/W)_{\Q} \hra \text{($(\Q \otimes \cO_{X_W})$-coherent left $(\Q \otimes \wh{\cD}_{X_W/W}^{(0)})$-modules)}, 
\end{equation} 
which is functorial with respect to $X_W/W$.

\medskip 

Next, we recall the basic facts on 
infinitesimal (iso)crystals. 
Basic references are 
\cite{Gro68}, \cite{ogusinf}, \cite[\S 2]{berthelotogus}. 
For a scheme $X$ of finite type over $k$, let 
$(X/W)_{\rm inf}$ (resp. $(X/W_n)_{\rm inf}$) 
be the {\it infinitesimal site} 
on $X/W$ (resp. $X/W_n$). An object 
is a nilpotent closed immersion $U \hra T$ over $W_n$ for some $n$ 
(resp. over $W_n$) from an open subscheme $U$ of $X$ to a scheme $T$. Morphisms are the obvious ones and the covering is defined in the same way as 
in the case of crystalline site. Thus 
$(X/W)_{\rm inf}$ contains  $(X/W_n)_{\rm crys}$ as a full subcategory.

\medskip 

We can define the structure sheaf $\cO_{X/W}$ (resp. $\cO_{X/W_n}$) and 
the notion of crystals of finite presentation on 
$(X/W)_{\rm inf}$ (resp. $(X/W_n)_{\rm inf}$), 
which we call {\it infinitesimal crystals} on $X/W$ (resp. $X/W_n$), 
in the same way as in the case of crystalline site. 
We denote the category of infinitesimal crystals on 
$X/W$ (resp. $X/W_n$) by ${\rm Inf}(X/W)$ (resp. ${\rm Inf}(X/W_n)$). 
The categories 
${\rm Inf}(X/W)$, ${\rm Inf}(X/W_n)$ also satisfy the descent property 
for Zariski coverings of $X$. 

\medskip 

The topos $(X/W)_{\rm inf}^{\sim}$ associated to 
$(X/W)_{\rm inf}$ is also  
functorial with respect to $X/W$. To prove it, 
we need to prove the analogue of \cite[Lem.~5.11, 5.12, 5.13]{berthelotogus}
for infinitesimal site. The proof of \cite[Lem.~5.11, 5.13]{berthelotogus}
works as it is (and we don't need the argument on PD-structure).
The proof of \cite[Lem.~5.12]{berthelotogus} works if we define 
$T$ there to be the $N$-th infinitesimal neighborhood of $U$ in 
$T_1 \times_Y T_2$ for $N \gg 0$. 
As a consequence, if we are given a diagram \eqref{eq:funct}, 
we have the morphism of topoi 
$(X'/W)_{\rm inf}^{\sim} \lra (X/W)_{\rm inf}^{\sim}$, 
the morphism of  
ringed topoi $((X'/W)_{\rm inf}^{\sim}, \cO_{X'/W}) \allowbreak \lra 
((X/W)_{\rm inf}^{\sim},\cO_{X/W})$ and 
the pullback functor 
$f^*: {\rm Inf}(X/W) \allowbreak \lra {\rm Inf}(X'/W)$. 
Similar functoriality holds also for 
the the ringed topos $((X/W_n)_{\rm inf}^{\sim}, \allowbreak \cO_{X/W_n})$ 
associated to $(X/W_n)_{\rm inf}$ and the category 
${\rm Inf}(X/W_n)$.  

\medskip 

As in the case of crystalline site, 
the natural inclusion of sites 
$(X/W_n)_{\rm inf} \hra (X/W)_{\rm inf}$  induces 
the restriction functor 
\begin{equation}\label{restinf}
{\rm Inf}(X/W) \lra {\rm Inf}(X/W_n); \,\,\, E \mapsto E_n, 
\end{equation}
which induces the equivalence 
\begin{equation}\label{projinf}
{\rm Inf}(X/W) \os{\simeq}{\lra} \varprojlim_n {\rm Inf}(X/W_n); 
\,\,\, E \mapsto (E_n)_n. 
\end{equation}
The functors \eqref{restinf}, \eqref{projinf} are also 
functorial with respect to $X/W$.
Also, for any object $U \hra T$ in 
$(X/W)_{\rm inf}$ (resp. $(X/W_n)_{\rm inf}$), 
the {\it  functor of evaluation} at $T$ 
\begin{equation}\label{evtinf}
{\rm Inf}(X/W) \lra \Coh(\cO_T) \,\,\,\text{(resp. 
${\rm Inf}(X/W_n) \lra \Coh(\cO_T)$)}; \,\,\, E \mapsto E_T
\end{equation}
is defined. 

\medskip 

When $X$ is a smooth variety over $k$ and there exists 
a lifting of $X$ to a $p$-adic smooth formal scheme $X_W$ over $W$,  
we have an equivalence of categories 
\begin{equation}\label{1.7inf}
{\rm Inf}(X/W_n) \os{\simeq}{\lra} 
\text{($\cO_{X_n}$-coherent left $\cD_{X_n/W_n}$-modules)}, 
\end{equation}
where $X_n := X_W \otimes_W W_n$ and 
$\cD_{X_n/W_n}$ is the full ring of differential operators of 
$X_n$ over $W_n$. 
Thus we have an equivalence of categories 
\begin{equation}\label{1.8inf}
{\rm Inf}(X/W) \os{\simeq}{\lra} 
\text{($\cO_{X_W}$-coherent left $\wh{\cD}_{X_W/W}$-modules)}, 
\end{equation}
where $\wh{\cD}_{X_W/W} := \varprojlim_n \cD_{X_n/W_n}$, and 
the action of $\wh{\cD}_{X_W/W}$ on  objects on the right hand side
 is assumed to be continuous. 
When there exists a local basis $x_1, ..., x_d$ of $X_W$ over $W$, 
$\wh{\cD}_{X_W/W}$ is topologically generated by 
$$\dfrac{1}{n!}\left(\dfrac{\partial}{\partial x}\right)^n := 
\dfrac{1}{n_1!}\left(\dfrac{\partial}{\partial x_1}\right)^{n_1} \dfrac{1}{n_2!}\left(\dfrac{\partial}{\partial x_2}\right)^{n_2} 
\cdots \dfrac{1}{n_d!}\left(\dfrac{\partial}{\partial x_d}\right)^{n_d} \quad (n := (n_1, ..., n_d) \in \N^d) $$
over $\cO_{X_W}$  \cite[Section 2]{berthelotogus}.
\medskip 

We give a proof of the equivalence \eqref{1.7inf}, which seems to be missing 
in the literature. For $n,m,r \in \N$, let $X_n(r)_m$ (resp. $X_n(r)'_m$) be the $m$-th infinitesimal 
neighborhood of $X$ (resp. $X_n$) in $X_n(r) := 
\underbrace{X_n \times_{W_n} \cdots \times_{W_n} X_n}_{r+1}$. Also, 
for $i=1,2$, let $p_{i,m}: X_n(1)_m \lra X_n$ (resp. $p'_{i,m}: X_n(1)'_m \lra X_n$) 
be the morphism induced by the $i$-th projection $X_n(1) \lra X_n$, and 
for $1 \leq i < j \leq 3$, let $p_{i,j,m}: X_n(2)_m \lra X_n(1)_m$ (resp. $p'_{i,j,m}: X_n(2)'_m \lra X_n(1)'_m$) 
be the morphism induced by the projection $X_n(2) \lra X_n(1)$ into the $i$-th and $j$-th factors. 
We denote by ${\rm Str}(X/W_n)$ (resp. ${\rm Str}(X_n/W_n)$) be 
the category of pairs $(E,\{\epsilon_m\}_m)$, where $E$ is a coherent 
$\cO_{X_n}$-module and $\{\epsilon_m: p_{2,m}^*E \os{\simeq}{\lra} p_{1,m}^*E\}_m$ 
(resp. $\{\epsilon_m: {p'}_{2,m}^*E \os{\simeq}{\lra} {p'}_{1,m}^*E\}_m$) 
is a compatible family of linear isomorphisms such that $\epsilon_0 = \id_{E}$ and 
$p_{1,2,m}^*\epsilon \circ p_{2,3,m}^*\epsilon = p_{1,3,m}^*\epsilon$ 
(resp. ${p'}_{1,2,m}^*\epsilon \circ {p'}_{2,3,m}^*\epsilon = {p'}_{1,3,m}^*\epsilon$). 
Such a datum is usually called a  {\it stratification} on $E$.  

\medskip 

Then, one has the functor 
\begin{align}
& {\rm Inf}(X/W_n) \lra {\rm Str}(X/W_n), \label{eq:str1} \\ 
& E \mapsto (E_{X_n}, \{ p_{2,m}^*E_{X_n} \os{\simeq}{\lra} E_{X_n(1)_m} \os{\simeq}{\lla} 
p_{1,m}^*E_{X_n}\}_m). \nonumber 
\end{align} 
We can also define the functor 
\begin{equation}\label{eq:strs}
{\rm Str}(X/W_n) \lra  {\rm Inf}(X/W_n)
\end{equation}
of converse direction as follows. If we are given $(E,\{\epsilon_m\}_m) \in {\rm Str}(X/W_n)$
and an object $U \hra T$ in $(X/W_n)_{\rm inf}$, we define the coherent 
$\cO_T$-module $E_T$ in the following way.
Since there exists a morphism 
$\varphi: T \lra X_n$ over $W_n$ lifting the closed immersion $X \hra X_n$ locally on $T$, 
$E_T$ is defined by $E_T := \varphi^*E$ 
locally. If we have two morphisms $\varphi, \varphi': T \lra X_n$ as above, 
$\varphi \times \varphi'$ induces a morphism 
$\psi: T \lra X_n(1)_m$ for some $m$, and $\psi^*\epsilon_m$ defines a 
gluing data for the sheaf $E_T$ defined locally as above. Thus $E_T$ is defined globally 
on $T$ by descent. Then the family $\{E_T\}_{U \hra T}$ gives an object 
of ${\rm Inf}(X/W_n)$. Thus the functor \eqref{eq:strs} is defined. One can check that 
it defines a quasi-inverse of \eqref{eq:str1}, and so \eqref{eq:str1} is an equivalence. 

\medskip 

Next, because the canonical closed immersions $X_n(r)_m \hra X_n(r)'_m$ induce 
the isomorphism $\{X_n(r)_m\}_m \os{\simeq}{\lra} \{X_n(r)'_m\}$ as 
ind-schemes, we have the canonical equivalence of categories 
\begin{equation}\label{eq:str2}
{\rm Str}(X_n/W_n) \os{\simeq}{\lra} {\rm Str}(X/W_n). 
\end{equation}
Finally, we have an equivalence 
\begin{equation}\label{eq:str3}
{\rm Str}(X_n/W_n) \os{\simeq}{\lra} \text{($\cO_{X_n}$-coherent left $\cD_{X_n/W_n}$-modules)}
\end{equation}
by \cite[Prop.~2.11, Rmk.~2.13]{berthelotogus}. By combining 
\eqref{eq:str1}, the quasi-inverse of \eqref{eq:str2} and \eqref{eq:str3}, 
we obtain the equivalence \eqref{1.7inf}. 

\medskip 

 By construction, the functors 
\eqref{eq:str1}, \eqref{eq:str2} and \eqref{eq:str3} are functorial 
with respect to $X_W/W$, namely, if we are given a diagram as \eqref{eq:functw}, 
we have the pullback by $f_W$ modulo $p^n$ on ${\rm Str}(X/W_n), {\rm Str}(X_n/W_n), 
$ and the category of $\cO_{X_n}$-coherent left $\cD_{X_n/W_n}$-modules, 
and the functors are compatible with 
respect to $f^*$ on ${\rm Inf}(X/W_n)$ and the above pullback functors. 
Thus the functors \eqref{1.7inf}, \eqref{1.8inf} are also functorial with respect to 
$X_W/W$.

\medskip 

For any infinitesimal crystal $E$ on $X/W$ or $X/W_n$, 
the value $E_T$ of $E$ at any $U \hra T$ is locally free.
To prove this, it suffices to consider the case of 
infinitesimal crystals on $X$ over $k$, and in this case, 
the claim follows from the equivalence \eqref{1.7inf}, 
Katz' theorem \cite[Thm.~1.3]{Gie75} and \cite[Lem.~6]{dossantos}. 

\medskip

For a smooth variety $X$ over $k$, 
let ${\rm Inf}(X/W)_{\Q}$ be the $\Q$-linearization of 
the category ${\rm Inf}(X/W)$, which is called the category of 
{\it infinitesimal isocrystals} on $X$. 
As in the case of crystalline site, 
one has a natural functor 
${\rm Inf}(X/W) \xrightarrow{\Q \otimes}  {\rm Inf}(X/W)_{\Q}$. 
When $X$ is liftable to a $p$-adic smooth formal scheme $X_W$ over $W$, 
the functor \eqref{1.8inf} induces the full embedding  
\begin{equation}\label{1.8infq}
{\rm Inf}(X/W)_{\Q} \hra 
\text{($(\Q \otimes \cO_{X_W})$-coherent left $(\Q \otimes \wh{\cD}_{X_W/W})$-modules)}, 
\end{equation} 
which is functorial with respect to $X_W/W$. 
For objects in the right hand side, 
the action of $\Q \otimes \wh{\cD}_{X_W/W}$ is assumed to be continuous.

\medskip 

Next we recall the basic facts on convergent isocrystals \cite{ogus1}, \cite{ogus2}. 
On a scheme $X$ of finite type over $k$,  the category ${\rm Enl}(X/W)$ of {\it enlargements}  is defined in \cite[Defn.~2.1]{ogus1}. Objects are pairs $(T, z_T)$ where $T$ is a $p$-adic formal flat scheme of finite type over ${\rm Spf}(W)$ together with a morphism $(T\otimes_W k)_{\rm red}\xrightarrow{z_T} X$. Morphisms in the category are the obvious ones. 
 A {\it convergent isocrystal} \cite[Defn.~2.7]{ogus1} on $X/K$ is a crystal on $ {\rm Enl}(X/W)$ with value on $(T, z_T)$  in the $\Q$-linearization  ${\rm Coh}(\cO_T)_{\Q}$ of the category ${\rm Coh}(\cO_T)$. This defines a category $\Conv(X/K)$ with the obvious morphisms, which is abelian \cite[Cor.~2.10]{ogus1}. We denote the structure convergent isocrystal on $X/K$ by $\cO_{X/K}$. The category $\Conv(X/K)$ is functorial 
with respect to $X/W$, and it has descent property for Zariski coverings.

\medskip

When $X$ is liftable to a $p$-adic smooth formal scheme $X_W$ over $W$, 
we have the equivalence 
\begin{equation}\label{1.8conv}
\Conv(X/K) \os{\simeq}{\lra} 
\text{($(\Q \otimes \cO_{X_W})$-coherent left $\cD^{\dagger}_{X_W,\Q}$-modules)}, 
\end{equation} 
where $\cD^{\dagger}_{X_W,\Q} = \varinjlim_m \Q \otimes \wh{\cD}^{(m)}_{X_W/W}$, 
$\wh{\cD}^{(m)}_{X_W/W} := \varprojlim_n \cD^{(m)}_{X_n/W_n}$ 
($X_n := X_W \otimes W_n$) and $\cD^{(m)}_{X_n/W_n}$ is the ring of PD-differential 
operators of level $m$ \cite[Prop.~4.1.4]{berthelotI}. 
It is functorial with respect to $X_W/W$.

\medskip 

For a smooth variety $X$ over $k$, Ogus defines in 
\cite[Thm.~0.7.2]{ogus2} a fully faithful functor 
\begin{equation}\label{eqn:convtocrys}
 \Phi: \Conv(X/K) \lra \Crys(X/W)_{\Q} 
 \end{equation}
using a nice system of objects in ${\rm Enl}(X/W)$ and 
the local nature of isocrystals \cite[Lem.~0.7.5]{ogus2},  
such that, for any $\cE \in \Conv(X/K)$, the 
convergent cohomology $H^i_{\conv}(X/K,\cE)$ (defined in 
\cite[Section~4]{ogus2}) and the crystalline cohomology 
$H^i_{\crys}(X/W,\Phi(\cE))$ coincide \cite[Thm.~0.7.7]{ogus2}.
The functor $\Phi$ is functorial with respect to $X/W$.  
Also, when $X$ is liftable to a $p$-adic smooth formal scheme $X_W$ over $W$, 
$\Phi$ is compatible with the canonical functor 
\begin{align}
& 
\text{($(\Q \otimes \cO_{X_W})$-coherent left $\cD^{\dagger}_{X_W,\Q}$-modules)} 
\label{eq:dd} \\ 
\lra \,\, & 
\text{($(\Q \otimes \cO_{X_W})$-coherent left $(\Q \otimes \wh{\cD}_{X_W/W}^{(0)})$-modules)}, 
\nonumber 
\end{align}
via \eqref{mic-dq} and \eqref{1.8conv}. The functor \eqref{eq:dd} is obviously functorial with respect to $X_W/W$.
In the following, we omit to write the functor 
$\Phi$ and regard a convergent isocrystal $\cE$ on 
$X$ as an isocrystal on $(X/W)_{\crys}$ via the functor $\Phi$. 
For the whole theory of convergent isocrystals, we also refer to  
\cite{berthelotpreprint}, \cite{lestum}.

\medskip 

 For a smooth variety $X$ over $k$,  
we have the full embedding of sites $(X/W_n)_{\rm crys} \hra (X/W)_{\rm inf}$, which 
induces the functors  
$$ 
\Phi': {\rm Inf}(X/W) \lra {\rm Crys}(X/W), \quad 
\Phi'_{\Q}: {\rm Inf}(X/W)_{\mathbb{Q}} \lra {\rm Crys}(X/W)_{\mathbb{Q}}. $$ 
They are functorial with respect to $X/W$. 
When $X$ is liftable to a $p$-adic smooth formal scheme $X_W$ over $W$, 
$\Phi'. \Phi'_{\Q}$ are compatible with the canonical functors 
\begin{align}
& 
\text{($\cO_{X_W}$-coherent left $\wh{\cD}_{X_W/W}$-modules)} 
\label{eq:dd2} \\ 
\lra \,\, & 
\text{($\cO_{X_W}$-coherent left $\wh{\cD}_{X_W/W}^{(0)}$-modules)}, 
\nonumber 
\end{align}
\begin{align}
& 
\text{($(\Q \otimes \cO_{X_W})$-coherent left $(\Q \otimes \wh{\cD}_{X_W/W})$-modules)} 
\label{eq:dd3} \\ 
\lra \,\, & 
\text{($(\Q \otimes \cO_{X_W})$-coherent left $(\Q \otimes \wh{\cD}_{X_W/W}^{(0)})$-modules)} 
\nonumber 
\end{align}
via \eqref{evformal}, \eqref{mic-d}, \eqref{mic-dq}, \eqref{1.8inf} 
\eqref{1.8infq}, 
because the constructions involved are done in a parallel way for 
(iso)crystals and infinitesimal (iso)crystals. The functors 
\eqref{eq:dd2} and \eqref{eq:dd3} are functorial with respect to $X_W/W$.

\medskip 

We prove that $\Phi'$ is fully faithful. 
To see this, we may work locally by Zariski descent for morphisms in  
${\rm Inf}(X/W)$ and ${\rm Crys}(X/W)$. 
So we may assume that $X$ lifts to a $p$-adic smooth formal scheme 
$X_W$ over $W$. Thus we are reduced to proving the full faithfulness of 
\eqref{eq:dd2}. Noting 
the local freeness of the values of any object in 
${\rm Inf}(X/W)$, we are reduced to proving the equality 
$$ M^{\wh{\cD}_{X_W/W} = 0} \os{\simeq}{\lra} M^{\wh{\cD}^{(0)}_{X_W/W} = 0} $$
of horizontal elements 
for any $\cO_{X_W}$-locally free $\wh{\cD}_{X_W/W}$-module $M$. 
This is clear because any such $M$ is flat over $W$ and 
the image of $\Q \otimes_{\Z} \wh{\cD}^{(0)}_{X_W/W}$ is dense in 
$\Q \otimes_{\Z} \wh{\cD}_{X_W/W}$ because, in terms of local coodinates 
$x := (x_1, ..., x_d)$, 
the former contains the sections $\dfrac{1}{n!}\left(\dfrac{\partial}{\partial x}\right)^n \,
(n \in \N^d)$ which topologically generates the latter.
As a consequence, we see that the functor $\Phi'_{\Q}$ is also 
fully faithful. 

\medskip

\medskip 
Also, we prove that the functor $\Phi'_{\Q}$ factors through 
$\Phi$ and thus induces the functor 
$$ {\rm Inf}(X/W)_{\Q} \lra \Conv(X/K) $$
which we denote also by $\Phi'_{\Q}$. To prove it, 
we may work locally by Zariski descent for $\Conv(X/K)$ and full faithfulness of 
$\Phi$. So we may assume that $X$ lifts to a $p$-adic smooth formal scheme 
$X_W$ over $W$, and in this case, the claim follows from the fact that 
the functor \eqref{eq:dd3} factors through \eqref{eq:dd}. 
In the following, we omit to write also the functor 
$\Phi'_{\Q}$ and regard an infinitesimal isocrystal on 
$X$ as a convergent isocrystal on $X/K$ 
(hence an isocrystal on $X$) via the functor $\Phi'_{\Q}$. 

\medskip

We recall the functoriality of the categories discussed above
with respect to the absolute Frobenius morphism $F: X \lra X$. 
By applying the functoriality with respect to the diagram \eqref{eq:funct} 
in the case $f = F$ and $f' = \sigma_W$, we obtain the pullback functors

\begin{align*}
\bullet \ & \Crys(X/W) \lra \Crys(X/W), \,\,\,\, 
\bullet \ \Crys(X/W_n) \lra \Crys(X/W_n), \\ 
\bullet \ & \MIC(X)^{\qn} \lra \MIC(X)^{\qn}, \,\,\,\, 
\bullet \ \MIC(X) \lra \MIC(X), \\ 
\bullet \ & 
\Coh(X) \lra \Coh(X), \,\,\,\, 
\bullet \ \Crys(X/W)_{\Q} \lra \Crys(X/W)_{\Q}, \\ 
\bullet \ & {\rm Inf}(X/W) \lra  {\rm Inf}(X/W), \,\,\,\, 
\bullet \ {\rm Inf}(X/W_n) \lra  {\rm Inf}(X/W_n), \\
\bullet \ &  {\rm Inf}(X/W)_{\Q} \lra  {\rm Inf}(X/W)_{\Q}, \,\,\,\,
\bullet \ \Conv(X/K) \lra \Conv(X/K), 
\end{align*}
which we all denote by $F^*$. By the functoriality discussed above, 
the functors 
\eqref{rest}, \eqref{proj}, \eqref{evxw}, \eqref{evxwn}, 
\eqref{evpn} for $n=1$, 
\eqref{restinf}, \eqref{projinf}, 
$\Phi$, $\Phi'$ $\Phi'_{\Q}$ 
are compatible with the various Frobenius pullbacks $F^*$. 
In particular, the pullback $F^*$ by Frobenius on $\Crys(X/W)_{\Q}$ respects 
the full subcategories $\Conv(X/K), {\rm Inf}(X/W)_{\Q}$.

\medskip 
When $X$ is liftable to a $p$-adic smooth formal scheme $X_W$ over $W$ and 
$F$ is liftable to a morphism $F_W: X_W \lra X_W$ over $\sigma_W$, 
$F_W$ and $\sigma_W$ induce 
the pullback functors 
{\allowdisplaybreaks{
\begin{align*}
\bullet \  & \MIC(X_n)^{\qn} \lra \MIC(X_n)^{\qn}, \,\,\,\, 
\bullet \ \MIC(X_W)^{\qn} \lra \MIC(X_W)^{\qn}, \\ 
\bullet \ & {\rm Str}(X/W_n) \lra {\rm Str}(X/W_n), \,\,\,\, 
\bullet \ {\rm Str}(X_n/W_n) \lra {\rm Str}(X_n/W_n), \\ 
\bullet \ &  \text{($\cO_{X_W}$-coherent left $\wh{\cD}_{X_W/W}^{(0)}$-modules)} \lra 
 \text{($\cO_{X_W}$-coherent left $\wh{\cD}_{X_W/W}^{(0)}$-modules)}, \\ 
\bullet \ & \text{($(\Q \otimes \cO_{X_W})$-coherent left $(\Q \otimes \wh{\cD}_{X_W/W}^{(0)})$-modules)} \\ & \lra 
 \text{($(\Q \otimes \cO_{X_W})$-coherent left $(\Q \otimes \wh{\cD}_{X_W/W}^{(0)})$-modules)}, \\ 
\bullet \ & \text{($\cO_{X_n}$-coherent left $\cD_{X_n/W_n}$-modules)} \lra 
\text{($\cO_{X_n}$-coherent left $\cD_{X_n/W_n}$-modules)}, \\ 
\bullet \ & \text{($\cO_{X_W}$-coherent left $\wh{\cD}_{X_W/W}$-modules)} 
\lra \text{($\cO_{X_W}$-coherent left $\wh{\cD}_{X_W/W}$-modules)}, \\ 
\bullet \ & \text{($(\Q \otimes \cO_{X_W})$-coherent left $(\Q \otimes \wh{\cD}_{X_W/W})$-modules)} \\ & \lra 
\text{($(\Q \otimes \cO_{X_W})$-coherent left $(\Q \otimes \wh{\cD}_{X_W/W})$-modules)}, \\
\bullet \ & \text{($(\Q \otimes \cO_{X_W})$-coherent left $\cD^{\dagger}_{X_W,\Q}$-modules)} \\ 
& \lra 
\text{($(\Q \otimes \cO_{X_W})$-coherent left $\cD^{\dagger}_{X_W,\Q}$-modules)}, 
\end{align*}}}
which we all denote by $F_W^*$. The functors 
\eqref{evpn}, \eqref{evformal}, \eqref{mic-d}, \eqref{mic-dq}, \eqref{1.7inf}, \eqref{1.8inf}, 
\eqref{eq:str1}, \eqref{eq:str2}, \eqref{eq:str3}, \eqref{1.8infq}, \eqref{1.8conv}, 
\eqref{eq:dd}, \eqref{eq:dd2}, \eqref{eq:dd3} are 
compatible with the various pullbacks by Frobenius or its lifting. 

\medskip

We give a short 
review on {\it Cartier descent} and the {\it inverse Cartier transform}
after Ogus-Vologodsky \cite{ov}. 

\medskip 

For $(E,\nabla) \in \MIC(X)$, one defines the $p$-{\it curvature}
$\beta: S^{\bullet}T_X \lra {\cal E}nd_{\cO_X}(E)$, which is 
an $F^*$-linear algebra homomorphism. We say that 
$(E,\nabla)$ has {\it zero} $p$-{\it curvature} (resp. 
has {\it nilpotent} $p$-{\it curvature of length} 
$p-1$) if $\beta(S^nT_X) = 0$ for 
$n \geq 1$ (resp. $n \geq p$). We denote the full subcategory of $\MIC(X)$
 of modules with integrable connection with 
zero $p$-curvature (resp. nilpotent $p$-curvature of length $p-1$) 
by $\MIC_{0}(X)$ (resp. $\MIC_{p-1}(X)$). The forgetful functor 
$\MIC_s(X)\to \Coh(X), \ s=0, p-1,$ yields the abelian structure on $\MIC_s(X)$ with respect to which the functor is exact.

\medskip
A {\it Higgs module} is a pair $(H, \theta)$ consisting of 
a coherent $\cO_X$-module $H$ and an $\cO_X$-linear morphism 
$\theta: H \lra H \otimes \Omega^1_X$ satisfying the 
integrability condition $\theta \wedge \theta = 0$. 
The map $\theta$, called the Higgs field,  induces an $\cO_X$-algebra homomorphism 
$S^{\bullet}T_X \lra {\cal E}nd_{\cO_X}(H)$, denoted by the same symbol $\theta$.
We say that $(H, \theta)$ has {\it nilpotent Higgs field of 
length} $\leq p-1$ if $\theta(S^nT_X) = 0$ for 
$n \geq p$. With the obvious morphisms, we denote the category of Higgs modules with 
Higgs field zero (resp. nilpotent Higgs field of length $\leq p-1$) by 
$\HIG_0(X)$ (resp. $\HIG_{p-1}(X)$). The forgetful functor 
$\HIG_{s}(X)\to \Coh(X), \ s=0,p-1,$ yields the abelian structure on $\HIG_{p-1}(X)$ with respect to which the functor is exact, and is an equivalence of categories for $s=0$.

\medskip

For a coherent $\cO_X$-module $E$, $F^*E$ is uniquely endowed with 
an integrable connection $\nabla_{\rm can}$ with zero $p$-curvature 
which is characterized by the condition that $F^{-1}(E) \subset F^*E$  is the subsheaf of abelian groups of flat sections.
The functor 
\begin{equation}\label{cartierd}
\Coh(\cO_X) \lra \MIC_0(X) , \ E \mapsto (F^*E, \nabla_{\rm can})
\end{equation}
 is an equivalence of 
categories. This fact is called {\it Cartier descent} (\cite[Section~2]{ov}).  
It is easy to see by direct calculation that the functor 
$F^*: \MIC(X) \lra \MIC(X)$ factors through \eqref{cartierd} and hence 
it has the form 
\begin{multline} \label{eqn:can}
F^*: \MIC(X)  \xrightarrow{(E,\nabla)\mapsto E} \Coh(X) \xrightarrow{\eqref{cartierd}} \MIC_0(X)\subset  \MIC(X),\\  (E,\nabla) \mapsto (F^*E, \nabla_{\rm can}).
\end{multline} 

\medskip

Ogus-Vologodsky generalized the equivalence \eqref{cartierd} when 
$X$ admits a smooth lifting $\wt{X}$ over $W_2$. 
(In \cite{ov}, they assume the existence of a smooth lifting $\wt{X}'$ of 
the Frobenius twist $X':=X\otimes_{\sigma} k$ over $W_2$, but this is equivalent to the condition above because $k$ is perfect.) Assuming the existence 
of $\wt{X}$, they generalized the 
equivalence \eqref{cartierd} to the equivalence 
\begin{equation}\label{cartierf}
C^{-1}: \HIG_{p-1}(X) \lra \MIC_{p-1}(X), 
\end{equation}
which is called {\it the inverse Cartier transform}. 
(Precisely speaking, the functor $C^{-1}$ here is the functor 
$C^{-1}_{\cX/\cS} \circ \pi^*_{X/S}$ used in \cite[(4.16.1)]{ov} 
for $\cX/\cS = (X,\wt{X}')/\Spec W_2$ and $S = \Spec k$.) 
Note that any object $(E,\nabla)$ in 
$\MIC(X)^{\qn}$ with $E$ torsion free of rank $\leq p$ is contained in 
$\MIC_{p-1}(X)$. 

\medskip

Finally, we recall some terminologies on (semi)stability. 
When $X$ is projective, we fix 
a $k$-embedding $\iota: X \hra \mathbb{P}_k^N$ of $X$ into a projective 
space and denote the pullback of $\cO_{\mathbb{P}_k^N}(1)$ to $X$ by 
$\cO_X(1)$. For a coherent, torsion free 
$\cO_X$-module $E$, the slope $\mu(E)={\rm deg}(E)/{\rm rank}(E)$ and
the reduced Hilbert polynomial $p_E(n)=\chi(X, E\otimes_{\cO_X} \cO_X(n))/{\rm rank}(E)$ are defined with respect to 
$\cO_X(1)$, as well as $\mu$- (Mumford-Takemoto) or  $\chi$- (Gieseker-Maruyama)  (semi)stability. As usual, $E\in \Coh(X)$ is 
said to be $\mu$-{\it  (semi)stable} if it is torsion free and $\mu(E')< \mu(E) \ (\mu(E')< \mu(E))$ for any strict subobject 
$0\neq E'\subset E$, and it is said to be {\it strongly} $\mu$-{\it semistable} if $(F^n)^*(E)$ is $\mu$-semistable for all natural numbers $n$. Similarly for $\chi$-(semi)stability.

\medskip

We  say that an object $E$ in $\Crys(X/k) = \MIC(X)^{\qn}$ 
is $\mu$-{\it semistable} as a crystal if $E$ is torsion free 
as an $\cO_X$-module and $\mu(E') \leq \mu(E)$ 
for any non-zero subobject $E'$ of $E$ in $\Crys(X/k) = \MIC(X)^{\qn}$. 
Similarly, we say that an object $(E,\nabla)$ in 
$\MIC_{p-1}(X)$ (resp. in $\HIG_{p-1}(X)$)  is 
$\mu$-{\it (semi)stable} if $E \in \Coh(X)$ is torsion free 
and $\mu(E') < \mu(E)  \ (\mu(E') \leq  \mu(E))$ 
for any  strict  subobject $0\neq (E',\nabla')$ of $(E,\nabla)$ in 
$\MIC_{p-1}(X)$  (resp. in  $\HIG_{p-1}(X)$). 

\medskip

\section{Statement of the main results}  \label{sec:results}

We say that $\cE \in \Crys(X/W)_{\Q}$ (resp. $\Conv(X/K)$, ${\rm Inf}(X/W)_{\Q}$, 
${\rm Inf}(X/k)$) is 
{\it constant} when it is isomorphic to a finite sum of 
the structure isocrystal $\Q \otimes \cO_{X/W}$ (resp. 
the structure convergent isocrystal $\cO_{X/K}$, the structure isocrystal 
$\Q \otimes \cO_{X/W}$, the strucure crystal $\cO_{X/k}$). 
As a $p$-adic version of  Gieseker's conjecture, according to which if a smooth projective variety $X$ over 
an algebraically closed field $k$ has a trivial \'etale fundamental group, then 
infinitesimal crystals on $X/k$, that is 
$\cO_X$-coherent $\cD_X$-modules, are constant 
(see \cite{esnaultmehta} for a positive answer),
 Johan de Jong posed the following conjecture in October 2010: 
 
\begin{conj}[de Jong]\label{dejongconj0}
Let $X$ be a connected smooth projective  variety over an algebraically closed field $k$ of characteristic $p>0$ and 
assume that the \'etale fundamental group of $X$ is 
trivial. Then any  isocrystal $\cE \in  \Crys(X/W)_{\Q}$   is 
constant. 
\end{conj}

By the fully faithful functor $\Phi: \Conv(X/K)\to  \Crys(X/W)_{\Q}$  defined in \eqref{eqn:convtocrys}, Conjecture~\ref{dejongconj0} contains the sub-conjecture

\begin{conj} \label{dejongconj}
Let $X$ be a connected smooth projective variety over an algebraically closed field $k$ of characteristic $p>0$ and 
assume that the \'etale fundamental group of $X$ is 
trivial. Then any convergent isocrystal $\cE \in  \Conv(X/K)$ is constant. 
\end{conj}

The aim of this article is to discuss  Conjecture~\ref{dejongconj}. 
We restrict our attention to Conjecture \ref{dejongconj} rather than Conjecture \ref{dejongconj0} mainly because 
 we shall strongly use the following proposition, which is 
a special case of \cite[Cor.4.10]{ogus1}: 
\begin{prop}\label{fdivision}
Let $X$ be a smooth variety over $k$. Then the pullback functor $F^*: \Conv(X/K)\to \Conv(X/K)$ is an equivalence of categories. 
\end{prop} 

Hence, for any convergent isocrystal 
$\cE$ on $X/K$ and for any $n \in \N$, 
there is the unique 
convergent isocrystal $\cE^{(n)}$ on 
$X/K$ with ${F^*}^n \cE^{(n)} = \cE$ up to isomorphism. We call this 
$\cE^{(n)}$ the $F^n$-{\it division} of $\cE$.

\begin{rem}\label{rmk:Fdivconv}
The category $\Conv(X/K)$ is characterized as the intersection 
$$\bigcap_n {\rm Im}(F^*)^n$$ of the essential images of $(F^*)^n \,(n \in \N)$ 
in $\Crys(X/W)_{\Q}$. The inclusion $\Conv(X/K) \subseteq 
\bigcap_n {\rm Im}(F^*)^n$ follows from Proposition \ref{fdivision}. 
To prove the inclusion in the other direction, we take  
$\cE \in \bigcap_n {\rm Im}(F^*)^n$. To prove that $\cE$ belongs to 
$\Conv(X/K)$, we may work locally, and so 
we may assume that $X$ is liftable to a $p$-adic smooth formal 
scheme $X_W$ over $W$ endowed with a lift $F_W: X_W \lra X_W$ of 
Frobenius morphism $F$ on $X$. Then $F_W$ induces an equivalence of 
categories 
{\allowdisplaybreaks{
\begin{align*}
(F_W^*)^n: & 
\text{($(\Q \otimes \cO_{X_W})$-coherent left $(\Q \otimes \wh{\cD}_{X_W/W}^{(0)})$-modules)} \\ & \os{\simeq}{\lra} 
\text{($(\Q \otimes \cO_{X_W})$-coherent left $(\Q \otimes \wh{\cD}_{X_W/W}^{(n)})$-modules)} 
\end{align*}}}
such that the composition of it with the canonical functor 
{\allowdisplaybreaks{
\begin{align*}
& \text{($(\Q \otimes \cO_{X_W})$-coherent left $(\Q \otimes \wh{\cD}_{X_W/W}^{(n)})$-modules)} \\ & \lra 
\text{($(\Q \otimes \cO_{X_W})$-coherent left $(\Q \otimes \wh{\cD}_{X_W/W}^{(0)})$-modules)} 
\end{align*}}} 
is equal to the pullback functor $(F_W^*)^n$ on the category of 
$(\Q \otimes \cO_{X_W})$-coherent left $(\Q \otimes \wh{\cD}_{X_W/W}^{(0)})$-modules, 
by \cite[Thm.~4.1.3, Rmq.~4.1.4(v)]{berthelotII}. 
From this and the compatibility of the functor \eqref{mic-dq} with  
$F^*$ and $F_W^*$, we see that the 
$(\Q \otimes \cO_{X_W})$-coherent left 
$(\Q \otimes \wh{\cD}_{X_W/W}^{(0)})$-module 
corresponding to $\cE$ admits an action of 
$\Q \otimes \wh{\cD}_{X_W/W}^{(n)}$ for any $n$. 
Such actions for $n \in \N$ are consistent because 
$\Q \otimes \wh{\cD}_{X_W/W}^{(0)}$ is dense in 
$\Q \otimes \wh{\cD}_{X_W/W}^{(n)}$. Thus the actions induce 
the action of $\cD^{\dagger}_{X_W,\Q} := \varinjlim_n 
\Q \otimes \wh{\cD}_{X_W/W}^{(n)}$. Hence $\cE$ belongs to  
$\Conv(X/K)$ by the equivalence \eqref{1.8conv}. 
\end{rem}

\begin{rem} \label{rmk:Fdivcrys}
Proposition~\ref{fdivision} does not necessarily 
extend to $\Crys(X/W)_{\Q}$ via \eqref{eqn:convtocrys} 
even when $X$ is projective and smooth, 
as is shown in Proposition~\ref{prop:fdivcrys} below. 
\end{rem}

\medskip

Given $\cE$ in $\Conv(X/K)$ or $\Crys(X/W)_{\Q}$, 
a crystal $E \in \Crys (X/W) $ is called a {\it lattice  of 
$\cE$}  if it is $p$-torsion free and
 $\cE = \Q \otimes E $ in $\Crys (X/W)_{\Q}$. 
Out of any choice  $E \in \Crys (X/W) $ with $\cE = \Q \otimes E $ in $\Crys (X/W)_{\Q}$, one constructs a lattice as follows. The  surjective morphisms $E/{\rm Ker}(p^{n+1}) \to 
E/{\rm Ker}(p^{n})$ in $\Crys(X/W)$ become isomorphisms for $n$ large. Indeed, as $X$ is of finite type, it is enough to show it on an affine $X$, for which one applies \eqref{evformal}.   Then $E/{\rm Ker}(p^{n})$ for $n$ large is a lattice of $\cE$. Clearly, there are then many lattices of the same $\cE$.

\medskip

We now formulate the first main result, 
proved in Section~\ref{sec:proofthm1} (compare with \cite[Thm.~1.7, Cor.~1.10]{Shi14}). 

\begin{thm}\label{thm1}
Let $X$ be a connected smooth projective variety over $k$ with
trivial  \'etale fundamental group. 
If $\cE \in \Conv (X/K)$ is such that for  any $n \in \N$, 
the $F^n$-division $\cE^{(n)}$ of $\cE$ 
admits a lattice $E^{(n)}$ with $E_X^{(n)} \in \Coh(X)$ 
strongly $\mu$-semistable, then $\cE$ is constant. 
\end{thm}


We have the following corollary confirming the conjecture of 
de Jong for infinitesimal isocrystals, which will be also proved in 
Section~\ref{sec:proofthm1}. 

\begin{cor} \label{cor}
Let $X$ be a connected smooth projective variety over $k$ with 
trivial \'etale fundamental group. Then any infinitesimal isocrystal on 
$X$ is constant. 
\end{cor}

A non-zero $\cE \in \Conv (X/K)$ is called {\it irreducible} 
if it is in its category (recall it is abelian), i.e. 
if it does not admit any non-zero strict subobject.  In general, every object admits  
a Jordan-H\"older filtration. Its irreducible subquotients are called {\it irreducible constituents}.
Using Theorem~\ref{thm1},  we shall  prove the 
following theorem in Section~\ref{sec:proofthm2}. 

\begin{thm}\label{thm2}
Let $X$ be a connected smooth projective  variety over $k$ with 
trivial \'etale fundamental group. If $\cE \in \Conv (X/K)$  satisfies either of the following conditions: 
\begin{itemize}
\itemsep0em 
\item[(1)]
 any irreducible constituent of $\cE$ is of rank $1$;
 \item[(2)] 
 $\mu_{\max}(\Omega^1_X) < 2$ and 
any irreducible constituent of $\cE$ is of rank $\leq 2$;
\item[(3)] $\mu_{\max}(\Omega^1_X) < 1$ and 
any irreducible constituent of $\cE$ is of rank $\leq 3$;
\item[(4)]  $r \geq 4$, $\mu_{\max}(\Omega^1_X) < \dfrac{1}{N(r)}$ and 
any irreducible constituent of $\cE$ is of rank $\leq r$, where 
$N(r) := \displaystyle\max_{a, b \geq 1, a+b \leq r} {\rm lcm}(a,b)$. 
\end{itemize}
then $\cE$ is constant. 
\end{thm}

\begin{cor} \label{rk1negativeslope}
 Let $X$ be a connected smooth projective  variety over $k$ with 
trivial \'etale fundamental group. Then
\begin{itemize}
\itemsep0em 
\item[(1)] any $\cE\in \Conv (X/K)$ of rank $1$ is constant;
\item[(2)] if $\mu_{\max}(\Omega^1_X) \leq 0$,  then any $\cE\in \Conv (X/K)$ is constant.
\end{itemize}
\end{cor}

By \cite[Thm.0.1]{Lan15b},  simply connected  non-uniruled  Calabi-Yau varieties of dimension $d$ in characteristic $\ge (d-1)(d-2) $ fulfill condition (2). However, even  if  many examples of  Calabi-Yau varieties which are not liftable to characteristic $0$ are known,  it is not clear whether some of them are not uniruled. If $X$ is simply connected and   uniruled, it is likely that one can show that $\cE\in \Conv (X/K)$ is constant directly by geometric method.

\medskip

When $p \geq 3$, there is a purely cohomological proof of 
Corollary~\ref{rk1negativeslope} (1). 
The following proposition is due to the first named author and Ogus. 
\begin{prop}\label{rem:eo}
 Let $X$ be a connected smooth projective  variety over $k$ with 
trivial \'etale fundamental group. Then 
\begin{itemize}
\itemsep0em 
\item[(1)]  for $p\ge 3$, any  locally free  $E\in \Crys (X/W) $ of rank $1$ is constant;
\item[(2)] extensions in  $\Crys (X/W)_{\Q}$ of $\Q \otimes \cO_{X/W}$ by itself are constant. 
\end{itemize}
\end{prop}

\begin{proof}
Since the \'etale fundamental group  is trivial, $H^1(X,\cO_X^\times)$ has no torsion 
of order prime to $p$. So ${\rm Pic}^0_{\rm red}(X) = 0$ and so 
$H^1_{\rm crys}(X/W, \cO_{X/W})= 0$. 
This proves (2). 

To prove (1), it suffices to prove 
the similar assertion for crystals on 
the nilpotent crystalline site $(X/W)_{\rm Ncrys}$, by \cite[IV~Rmq.~1.6.6]{berthelotbook}.

One has the  exact sequences  on $(X/W)_{\rm Ncrys}$, 
\begin{align}
& 1 \lra K \lra \cO_{X/W}^\times \lra \cO_X^\times \lra 1, \label{eq:eo0} \\ 
& 0 \lra J \lra \cO_{X/W} \lra \cO_X \lra 0, \label{eq:eo-1} 
\end{align}
defining 
 $J$  and $K$. The $p$-adic logarithm
$\log: K \ra J$ and exponential $\exp: J \ra K$ functions are well defined and  are isomorphisms on $(X/W)_{\rm Ncrys}$. Hence, 
from the exact sequences \eqref{eq:eo0}, \eqref{eq:eo-1}, 
we obtain on $(X/W)_{\rm Ncrys}$ the exact sequences 
\begin{equation}
0\to 
 H^1_{\rm Ncrys}(X/W, J) \to H^1_{\rm Ncrys}(X/W, \cO_{X/W}^\times) 
\os{\alpha}{\ra} H^1(X,\cO_X^\times) \os{\beta}{\ra} H^2_{\rm Ncrys}(X/W, J),
\label{eq:eo2} 
\end{equation}
and 
\begin{equation} 
0 \ra H^1_{\rm Ncrys}(X/W, J) 
\os{\gamma}{\ra} H^1_{\rm Ncrys}(X/W, \cO_{X/W}). \label{eq:eo1} 
\end{equation}

From \eqref{eq:eo1} and the vanishing $H^1_{\rm Ncrys}(X/W, \cO_{X/W}) = 
H^1_{\crys}(X/W, \cO_{X/W}) = 0$ (where the first equality follows from \cite[V~2.4]{berthelotbook}), one deduces
$H^1_{\rm Ncrys}(X/W, J) \allowbreak = 0$. Hence $\alpha$ is injective. 
Moreover, by 
\eqref{eq:eo2}, we obtain the commutative 
diagram 
\begin{equation}\label{eq:eo3}
\begin{CD}
H^1_{\rm Ncrys}(X,\cO_{X/W}^\times) 
@>{\alpha}>> H^1(X,\cO_X^\times) @>{\gamma \circ \log \circ \beta}>> 
H^2_{\rm Ncrys}(X/W, \cO_{X/W}) \\ 
@. @V{\delta}VV @| \\ 
@. H^1(X,\cO_X^\times) \otimes \Z_p @>{\epsilon}>> 
H^2_{\rm Ncrys}(X/W, \cO_{X/W}), 
\end{CD}
\end{equation}
where $\delta$ is induced by the inclusion $\Z \lra \Z_p$ and 
$\epsilon$ is the $\Z_p$-linearization of 
$\gamma \circ \log \circ \beta$. 
By construction, 
the upper horizontal line is a complex. On the other hand, 
$\delta$ is injective because $H^1(X,\cO_X^\times )$ has no 
torsion of order prime to $p$. Also, by the commutativity of the 
diagram \cite[I~(5.1.7)]{gros}, $\epsilon$ is identified with 
the map ${\rm NS}(X) \otimes \Z_p \lra H^2_{\rm crys}(X/W, \cO_{X/W})$ 
given in \cite[II Prop.~6.8]{illusie} and so it is injective. 
Then an easy diagram chase shows that 
$H^1_{\rm Ncrys}(X,\cO_{X/W}^\times) = 0$. This proves (1).
\end{proof}

Proposition \ref{rem:eo} together with the following proposition 
implies Corollary \ref{rk1negativeslope} (1) when $p \geq 3$. 

\begin{prop}\label{prop:rk1lattice}
Let $X$ be a smooth variety over $k$. Then 
any isocrystal on $X$ which is filtered such that 
the associated graded isocrystal is a sum of rank $1$ isocrystals
  admits a locally free lattice. 
\end{prop}

\begin{proof} We start with the rank $1$ case.
The idea of the proof is simple. Locally  it uses the equivalence  \eqref{evformal} and the fact that the reflexive hull of a rank $1$ coherent sheaf on a regular scheme is locally free. 

So assume first $X$ lifts to an affine $p$-adic smooth formal scheme 
$X_W={\rm Spf}(A)$ (hence $X={\rm Spec} (A/p)$).
Let $E$ be a lattice of a rank $1$ isocrystal $\cE$. Via \eqref{evformal}, writing $M=\Gamma(X_W, E)$, $E\in \Crys(X/W)$ is given by an integrable quasi-nilpotent connection 
\begin{equation} \label{conn}
\nabla_M: M\to M\otimes_A \hat \Omega^1_{A}
\end{equation}
 where $\hat \Omega^1_{A}:=\varprojlim_n  \Gamma(X_n, \Omega^1_{X_n})$.
Then $A, \  \hat \Omega^1_A, \ M $ are sheafified to $\cO,  \ \Omega^1, \ \cM$ 
on $\Spec(A)$ by the standard procedure. 

As $M$ is $p$-torsion free and 
$A[p^{-1}]\otimes_A M$ is a projective module over $A[p^{-1}]$ of rank $1$, there is a closed codimension $\ge 2$ subscheme $C\subset \Spec (A)$ such that $\cM$ restricted to $\Spec(A) \setminus C$ is locally free of rank $1$.
Define $N=\Gamma(\Spec(A) \setminus C, \cM)$. Then, as $A$ is regular, 
$N$ is a projective $A$-module of rank $1$ and, as $\hat \Omega^1_A$  is a projective $A$-module, projection formula implies that \eqref{conn} induces 
 an integrable connection
\begin{equation} \label{conn:N}
\nabla_{N}:  N \to N \otimes_A \hat \Omega^1_{A}
\end{equation}
which is quasi-nilpotent, as this condition is local on sections of $\cM$.
The connection \eqref{conn:N} can be seen as a connection on $X_W = \Spf(A)$, and so 
it defines a locally free lattice $E$ in $\Crys(X/W)$ of $\cE$. 

Then we can glue the locally defined lattices of $\cE$ by Lemma \ref{lem:rk1lattice} below, 
by replacing the lattices $E$ by $p^nE$ for suitable $n$'s.  This finishes the proof of the rank $1$ case. 

Let $0\to \cE'\xrightarrow{a}  \cE\xrightarrow{b} \cE''\to 0$ be an exact sequence in $\Crys(X/W)_{\Q}$, with $\cE\neq 0, \cE''\neq 0$, both $\cE"$, $\cE''$ satisfying the assumptions of the proposition. Let $E$ be a lattice of $\cE$. Then $0\to a^{-1}({\rm Ker}(b)) \xrightarrow{a} E \xrightarrow{b} b(E)\to 0$ is an exact sequence $\epsilon$ in $\Crys(X/W)$, while $a^{-1}({\rm Ker}(b))$
 is a lattice of $\cE'$, and $b(E)$  is a lattice of $\cE''$. By induction on the rank of $\cE$, there are locally free lattices $E'$ of $\cE'$ and $E''$ of $\cE''$, which we can rescale by multiplication with $p$-powers such that they are injective maps $E''\xrightarrow{i} b(E)$ and $a^{-1}({\rm Ker}(b)) \xrightarrow{j} E'$. If we pull back $\epsilon$ by $i$ and push the 
resulting extension by $j$, we obtain the extension 
$0\to E'\to E''' \to E''\to 0$ of $E''$ by $E'$  in $\Crys(X/W)$ such that $\Q\otimes E'''=\cE$. Moreover, $E'''$ is locally free. This finishes the proof.
\end{proof}

\begin{lem}\label{lem:rk1lattice}
Let $X$ be a connected smooth variety over $k$ and let 
$E, E'$ be rank $1$ locally free crystals on $(X/W)_{\crys}$ endowed with 
an isomorphism 
$\varphi: \Q \otimes E \os{\cong}{\lra} \Q \otimes E'$ in 
$\Crys(X/W)_{\Q}$. 
Then, for some $n \in \Z$, $\varphi$ induces an isomorphism 
$p^n E \os{\cong}{\lra} E'$ in $\Crys(X/W)$. 
\end{lem}

\begin{proof}
By assumption, $\varphi$ induces the isomorphism 
$\Q \otimes (E^\vee\otimes E') \os{\cong}{\lra} \Q \otimes \cO_{X/W}$ in 
$\Crys(X/W)_{\Q}$. 
Thus $H^0_{\crys}(X/W, E^\vee\otimes E')={\rm Hom}_{\Crys (X/W)}(E, E')= p^{n}W$ for some 
$n \in \Z$. Hence $\varphi$ induces an invertible morphism 
$p^{n} E \lra E'$, thus an isomorphism. 
\end{proof}

We shall reprove in Theorem~\ref{thm:lf}  by a different method a weaker version of Proposition~\ref{prop:rk1lattice},  
together with  statements in the  higher rank case.
\medskip 

Finally in this section, we provide an example for which $F^*$ is not surjective on $\Crys(X/W)_{\Q}$, 
which we promised in Remark \ref{rmk:Fdivcrys}. 

\begin{prop}\label{prop:fdivcrys}
Assume $p \geq 3$ and let $X$ be a supersingular elliptic curve. 
Then the pullback functor 
$$ F^*: \Crys(X/W)_{\Q} \lra \Crys(X/W)_{\Q} $$
is not essentially surjective. 
\end{prop}

\begin{proof}
Let $X_W$ be a formal lift of $X$ over $\Spf W$ 
and let $\omega$ be a generator of 
$H^1(X_W, \Omega^1_{X_W/W})$. Then 
$(\cO_{X_W}, d + p \omega)$ defines a module with integrable connection on $X_W$ over $W$ and one can check that it is quasi-nilpotent. 
Hence it defines a non-constant locally free crystal of rank $1$ on 
$(X/W)_{\crys}$, which we denote by $E$. 
We prove 
\begin{claim} \label{claim:f}
$(\cO_{X_W}, d + p \omega) \in \Crys(X/W)_\Q$ is not infinitely $F^*$-divisible.

\end{claim}
\begin{proof}
By Lemma \ref{lem:rk1lattice}, when 
$E'$ is a locally free crystal of rank $1$ on 
$(X/W)_{\crys}$ such that $\Q \otimes E \cong \Q \otimes E'$, 
then $E$ and $E'$ are isomorphic crystals. 
So, if $E$ is infinitely $F^*$-divisible in $\Crys(X/W)_{\Q}$, then so in $\Crys(X/W)$. 
Hence it suffices to 
prove that $E$ is not infinitely $F^*$-divisible in $\Crys(X/W)$. 

By the restriction functor from $\Crys(X/W)$ to the category of crystals   
on $(X/W)_{\rm Ncrys}$, it is enough to show that $E$ on
$(X/W)_{\rm Ncrys}$ is not infinitely $F^*$-divisible.

Assume that $E=(F^*)^nE'_n$ 
with $E'_n \in H^1_{\rm Ncrys}(X/W,\cO_{X/W}^{\times})$ for $n \in \N$. 
Then $(E'_n)_X\in H^1(X,\cO_X^{\times})$ is a $p^n$-torsion line bundle, thus is constant as $X$ is supersingular. Thus via \eqref{eq:eo2}, $E, E'_n \in H^1_{\rm Ncrys}(X/W, J)$, and via \eqref{eq:eo1}, 
$0\neq \gamma(E)=(F^*)^n\gamma(E'_n) \in H^1_{\rm Ncrys}(X/W, \cO_{X/W}) 
=H^1_{\rm crys}(X/W, \cO_{X/W})$. 
As the slopes of $F^*$ on $H^1_{\rm crys}(X/W, \cO_{X/W})$ are strictly positive, this is impossible. 

\end{proof}
\end{proof}

\section{Proof of Theorem \ref{thm1}} \label{sec:proofthm1}

In this section, we prove Theorem \ref{thm1}, so throughout, $X$ is a smooth projective variety  of dimension $d$ over an algebraically closed field $k$ of characteristic $p>0$. 

\medskip

A coherent $\cO_X$-module $E$ has crystalline Chern classes $c_i^{\crys}(E)$ in crystalline cohomology $
H^{2i}_{\crys}(X/W)$, a module of finite type over $W$. 
In \cite[\S 1.1]{langersfund}, numerical Chern classes $c_i(E)$  are defined in  the group $Z_{d-i}(X)/\!\sim$,  where $\sim$ is the  numerical  equivalence relation on the free group on dimension $(d-i)$-points.  
Denoting by $H^i_{{\rm alg}}\subset H^{2i}_{\crys}(X/W)$ the sub-$W$-module spanned by the $c_i^{\crys}(E)$'s, one has a group homomorphism $H^i_{{\rm alg}}\to (Z_{d-i}(X)/\!\sim) \otimes_{\Z} W$. Since $Z_{d-i}(X)/\!\sim$ is a free $\Z$-module of finite rank, $c_i^{\crys}(E)=0$ implies $c_i(E)=0$. As the (reduced) Hilbert polynomial depends only on $c_i(E)$, if $c_i(E)=0$ for all $i\ge 1$, then $p_E=p_{\cO_X}$.

\begin{prop}\label{prop:chern}
Let $X$ be a smooth projective  variety over $k$, 
let $\cE$ be a convergent isocrystal on $X/K$ and 
let $E \in \Crys(X/W)$ be a lattice of $\cE$. Then 
$c_i^{\crys}(E_X)=0$ for any $i > 0$, and, if $E_X$ is torsion free, 
 $p_{E_X}=p_{\cO_X},$   $\mu(E_X) = 0$.
\end{prop}

\begin{proof}
For $n \in \N$, let $\cE^{(n)}$ be the $F^n$-division of 
$\cE$ and let $E^{(n)} \in \Crys(X/W)$ be a lattice of $\cE^{(n)}$. 
As $X$ is smooth, there exists locally a lift $X_W$ of $X$ to a smooth $p$-adic formal scheme over $W$ and 
a local lift of $F$ on $X_W$, which is faithfully flat. 
Thus the equivalence \eqref{evformal} implies that 
the $p$-torsion freeness of a crystal is preserved by $F^*$ and so 
$(F^*)^nE^{(n)}$ is a lattice of $\cE$. In addition 
$((F^*)^nE^{(n)})_X = (F^*)^nE^{(n)}_X$. Hence, if we  prove 
$c^{\crys}_i(E_X) = c^{\crys}_i((F^*)^nE^{(n)}_X)$ in this situation, 
we have $c_i^{\crys}(E_X) = c_i^{\crys}((F^*)^nE^{(n)}_X) = p^{ni} c_i^{\crys}(E^{(n)}_X)$ for all $n\ge 1$, thus $c_i^{\crys}(E_X) = 0$ in the finite type $W$-module 
$H^{2i}_{\crys}(X/W)$, as claimed. Therefore, 
it suffices to prove that $c^{\crys}_i(E_X)$ does not depend on the 
choice of the lattice $E$. \par 
So let $E'$  be another lattice of $\cE$. Then, 
replacing $E'$ by $p^aE'$ for some $a\in \N$, one has 
$p^nE \subseteq E' \subseteq E$ for some $n \in \N$, where 
$p^nE$ is the image of $p^n: E \lra E$, and it suffices to treat this case. 
For $0 \leq i \leq n$, let ${E'}^i$ be the image of the map 
$E' \oplus p^iE \lra E$ defined as the sum of inclusions. 
Then we have ${E'}^0 = E, {E'}^n = E'$ and 
$p{E'}^{i-1} \subseteq {E'}^i \subseteq {E'}^{i-1} \, (1 \leq i \leq n)$. 
So to prove the equality $c^{\crys}_i(E) = c^{\crys}_i(E')$, we may assume that 
$pE \subseteq E' \subseteq E$. If we denote 
$Q := \Coker (E' \ra E)$, we have the following commutative diagram with 
exact horizontal lines in $\Crys(X/W)$: 
\begin{equation*}
\begin{CD}
0 @>>> E' @>>> E @>>> Q @>>> 0 \\ 
@. @VpVV @VpVV @V0VV \\ 
0 @>>> E' @>>> E @>>> Q @>>> 0. 
\end{CD}
\end{equation*}
By the snake lemma, we obtain the exact sequence 
$$ 0 \lra Q \lra E'/pE' \lra E/pE \lra Q \lra 0 $$ 
in $\Crys(X/W)$. Since all the objects in the above sequence are 
$p$-torsion, we can regard it as an exact sequence in $\Crys(X/k)$. 
By evaluating this sequence at $X$ and noting the equalities 
$(E/pE)_X = E_X/pE_X = E_X, (E'/pE')_X = E'_X/pE'_X = E'_X$, 
we obtain the exact sequence 
$$ 0 \lra Q_X \lra E'_X \lra E_X \lra Q_X \lra 0 $$
of coherent $\cO_X$-modules. Hence 
$[E_X] = [E'_X]$ in $K_0(X)$ and so $c^{\crys}_i(E_X) = c^{\crys}_i(E'_X)$ for all 
$i \in \N$. 
\end{proof}

The following proposition, which uses Gieseker's conjecture, 
proven in \cite{esnaultmehta}, 
is the key step for the proof.

\begin{prop}\label{prop:a} 
Let $X$ be a connected smooth projective variety over $k$ 
with trivial \'etale fundamental group. Let $r$ be a positive integer. 
Then there exists a positive integer $a = a(X,r)$ 
satisfying the following condition$:$ 
For any sequence of $\chi$-stable sheaves $\{E_i\}_{i=0}^a$ on 
$X$ with ${\rm rank}\,E_0 \leq r$, 
$p_{E_i} = p_{\cO_X} \,(0 \leq i \leq a)$ and 
$F^*E_i = E_{i+1} \,(0 \leq i \leq a-1)$, 
$E_a$ is necessarily of rank $1$ and isomorphic to $\cO_{X}$. 
\end{prop}

\begin{proof} 
By standard base change argument, we may assume that $k$ is uncountable. 
For $1 \leq s \leq r$, let $M_s$ be the moduli 
of $\chi$-stable sheaves on $X$ with rank $s$ and reduced 
Hilbert polynomial 
$p_{\cO_X}$, which is constructed by Langer (\cite{langerpositive}, 
\cite{langermixed}). 
It is a quasi-projective scheme over $k$. Also, let $M_{s,n}$ be 
the locus of closed  points consisting of $\chi$-stable sheaves 
$G$ such that $(F^*)^nG$ remains 
$\chi$-stable. It is known to be an open subvariety of  $M_s$ endowed with the reduced structure. (See discussion in 
the beginning of \cite[\S 3]{esnaultmehta}.) 
The pullback by $F$ induces the morphism $V$ (over $\sigma$) called Verschiebung
$$ \cdots \lra M_{s,2} \os{V}{\lra} M_{s,1} \os{V}{\lra} M_s. $$
Let $\im V^n$ be the image of $V^n: M_{s,n} \lra M_s$, 
which is a constructible set of the topological space $M_s$. Then, $\dim \im V^n$ 
is stable for $n \gg 0$, which we denote by $f$. 
Assume $f > 0$. Then the generic point of 
some irreducible closed subscheme of dimension $f$ remains contained in 
$\im V^n \, (n \in \N)$. 
Pick such an irreducible closed subscheme and denote it by $C$. 
Then $C \cap \im V^n$ is non-empty for any $n \in \N$ and it contains 
an open subscheme of $C$. So there exists a closed 
subscheme $D_n \subsetneq C$ of smaller dimension 
such that $C \cap \im V^n \supseteq C \setminus D_n$. Then 
$C \cap (\bigcap_n \im V^n) \supseteq C \setminus (\bigcup_n D_n)$, and 
$C \setminus (\bigcup_n D_n)$
contains at least two $k$-rational points $P,P'$, 
because $k$ is uncountable. 
On the other hand,  the $k$-rational points of $\bigcap_n \im V^n$  are moduli points of
 infinitely $F$-divisible torsion free sheaves, thus they are locally free  infinitely $F$-divisible sheaves.  By the affirmation \cite[Thm.~1.1]{esnaultmehta} of Gieseker's conjecture, $\bigcap_n \im V^n$ is either empty or consists only of the moduli point of $\cO_X$. 
Since $P,P'$ are different $k$-rational points of $\bigcap_n \im V^n$, 
this is a contradiction. So $\im V^n$ consists of  a
finite set of points (possibly empty) for some $n$. 
Then, since $\bigcap_n \im V^n$ is empty (if $s \geq 2$) or  consists of 
one point corresponding to $\cO_{X}$ 
(if $s=1$),  
it is equal to $\im V^{a(s)}$ for some $a(s) \in \N$. 
Let us define $a$ to be the maximum of natural numbers $a(s)$ $(s \leq r)$. 
Then, if we are given a 
sequence $\{E_i\}_{i=0}^a$ 
as in the statement of the proposition 
with $s := {\rm rank} E_0 \leq r$, 
$E_a$ defines a $k$-rational point of $\im V^{a(s)} \subseteq M_s$. 
Then $s$ should be equal to $1$ and 
$E_a$ should be isomorphic to 
$\cO_{X}$. 
\end{proof}

We also use the following proposition. 

\begin{prop}\label{prop:b}
Let $X$ be a connected projective smooth variety over $k$ 
with  trivial \'etale fundamental group. 
Then 
there exists a positive integer 
$b = b(X)$ satisfying the following condition$:$ 
For any sequence of locally free sheaves $\{E_i\}_{i=0}^{b(r-1)}$ 
on $X$ with ${\rm rank}\,E_0 = r$, 
$F^*E_i = E_{i+1} \,(0 \leq i \leq b(r-1)-1)$ such that 
$E_{0}$ is an iterated extension of $\cO_{X}$, 
$E_{b(r-1)}$ is isomorphic to $\cO_{X}^r$. 
\end{prop}

\begin{proof}
The proof is similar to that in \cite[Prop.~2.4]{esnaultmehta}. 
By \cite[Cor.~on~p.143]{mumford}, one has the decomposition 
$$H^1(X,\cO_{X}) = H^1(X,\cO_{X})_{\rm nilp} \oplus 
H^1(X,\cO_{X})_{\rm ss}$$ of $H^1(X,\cO_{X})$ as $k$-vector spaces 
such that the absolute Frobenius $F^*$ acts nilpotently on 
$H^1(X,\cO_{X})_{\rm nilp}$ 
  and as a bijection on 
$H^1(X,\cO_{X})_{\rm ss}$. 
Moreover, one has 
\begin{align} 
H^1(X,\cO_{X})_{\rm ss} & = 
H^1(X,\cO_X)^{F=1} \otimes_{\F_p} k \label{eqn:O} \\ & = 
H^1_{\et}(X,\F_p) \otimes_{\F_p} k = 
\Hom(\pi_1^{\et}(X),\F_p) \otimes_{\F_p} k = 0, \nonumber 
\end{align}
and there exists some $b \in \N$ 
such that $(F^*)^b$ acts by $0$ on 
$H^1(X,\cO_{X})_{\rm nilp}$, 
since $H^1(X,\cO_{X})_{\rm nilp}$ is finite-dimensional. 
So $(F^*)^b$ acts by $0$ on $H^1(X,\cO_{X})$. 
We prove the proposition for this choice of $b$. 

By assumption on $E_{0}$, there exists a filtration 
$$ 0 = E_{0,0} \subset E_{0,1} \subset  
\cdots \subset E_{0,r} = E_{0} $$
the graded quotients of which are isomorphic to $\cO_{X}$. 
By pulling back to $E_{i}$ via $(F^*)^{i}$, we obtain the 
filtration 
$$ 0 = E_{i,0} \subset E_{i,1} \subset  \cdots \subset E_{i,r} = E_{i} $$
of $E_i$ still with graded quotients   isomorphic to $\cO_{X}$. 
We prove that $E_{b(\ell-1),\ell}$ is isomorphic to 
$\cO_{X}^\ell$ by induction on $\ell$. 
Assume that $E_{b(\ell-1),\ell} \cong \cO_{X}^\ell$. 
Then, for $b(l-1) \leq n \leq bl$, 
consider the extension class $e_{n}$ of the exact sequence 
$$ 0 \lra E_{n,\ell} \lra E_{n,\ell+1} \lra \cO_{X} \lra 0 $$
in $H^1(X,E_{n,\ell}) = H^1(X,\cO_{X})^\ell$. 
The family of classes $\{e_n\}_{n=b(\ell-1)}^{b\ell}$ defines an element of 
the inverse limit of the diagram 
$$ H^1(X,\cO_{X})^\ell \os{F^*}{\lra} 
\cdots 
\os{F^*}{\lra} H^1(X,\cO_{X})^\ell $$
of length $b$ with last component  $e_{b\ell}$. 
Then, by definition of $b$, $e_{b\ell} = 0$. 
So $E_{b\ell,\ell+1}$ is isomorphic to $\cO_X^{\ell+1}$. This finishes the proof. 
\end{proof}

We use Propositions~\ref{prop:a} and \ref{prop:b} to proceed towards the proof of Theorem~\ref{thm1}. 
\begin{prop}\label{prop:c}
Let $X$ be a connected smooth projective variety over $k$ 
with trivial \'etale fundamental group. Let $r$ be a positive integer. 
Let $\cE \in \Conv(X/K)$ be  of rank $r$ 
and let $E$ be a lattice of $\cE$ such that 
$E_X \in \Coh(X)$  is strongly $\mu$-semistable. 
Then, there exists a positive integer $c = c(X,r)$ 
such that  $((F^*)^cE)_X \in \Crys(X/k)$
 is constant.
\end{prop}

The following result of Langer \cite[Theorem 4.1]{langersfund} plays 
a crucial r\^ole in the proof of Proposition~\ref{prop:c}. 

\begin{thm}[Langer]\label{thm:langer}
Let $X$ be  a smooth projective  variety over $k$, 
let $E$ be a strongly $\mu$-semistable sheaf 
with vanishing Chern classes. 
Then there exists a filtration of 
$E$ whose graded quotients are $\mu$-stable, strongly $\mu$-semistable 
locally free sheaves with vanishing Chern classes. 
\end{thm}

\begin{proof}[Proof of Proposition \ref{prop:c}]
Take $a = a(X,r), b = b(X) \in \N$ so that the statement of 
Propsitions~\ref{prop:a}, \ref{prop:b} are satisfied, and 
set  $c := br(r-1)+ar+1$. We prove that the proposition is true for 
this choice of $c$. 

First, note that $(F^*)^nE_X = ((F^*)^nE)_X$ for any $n \in \N$. 
Hence $(F^*)^nE_X$  has vanishing Chern classes 
by Proposition~\ref{prop:chern}, and is strongly $\mu$-semistable by 
assumption, for all $n\ge 0$.

For $0 \leq n \leq c-1 = br(r-1)+ar$, 
define a filtration 
$\{((F^*)^nE_X)_q\}_{q=0}^{q_n}$ of $(F^*)^nE_X$ 
whose graded quotients are $\mu$-stable, strongly $\mu$-semistable 
with vanishing Chern classes, in the following way.
First, when $n=0$, take a filtration 
$\{(E_X)_q\}_{q=0}^{q_0}$ of $E_X$ 
whose graded quotients are $\mu$-stable, strongly $\mu$-semistable 
with vanishing Chern classes. 
(Such a filtration exists by Theorem \ref{thm:langer} because 
$E_X$ has vanishing Chern classes 
by Proposition \ref{prop:chern} and strongly $\mu$-semistable by 
assumption.) 
When we defined 
$\{((F^*)^{n-1}E_X)_q\}_{q=0}^{q_{n-1}}$, 
the pull-back 
$\{F^*((F^*)^{n-1}E_X)_q\}_{q=0}^{q_{n-1}}$ of it by $F^*$ 
defines a filtration of $(F^*)^nE_X$ whose graded quotients 
are strongly $\mu$-semistable with vanishing Chern classes. 
Then, using Theorem \ref{thm:langer} for the graded quotients, 
we can refine this filtration to a filtration 
$\{((F^*)^nE_X)_q\}_{q=0}^{q_n}$ of $(F^*)^nE_X$ 
whose graded pieces are $\mu$-stable, strongly $\mu$-semistable 
with vanishing Chern classes. 

By definition, we have 
$$ 1 \leq q_0 \leq q_1 \leq \cdots \leq q_{c-1} = q_{br(r-1)+ar} \leq r. $$
So, there exists some 
$j$ with $0 \leq j \leq b(r-1)^2+a(r-1)$ 
such that $q_{j} = \cdots = q_{j+b(r-1)+a} \allowbreak (=:Q)$. 

Put $G_{n,q} := ((F^*)^nE_X)_q/((F^*)^nE_X)_{q-1}$. Then, 
for each $1 \leq q \leq Q$, 
any subsequence of length $a$ of 
the sequence $\{G_{n,q}\}_{n=j}^{j+b(r-1)+a}$ 
satisfies the assumption of Proposition~\ref{prop:a}. 
Hence 
$\{G_{n,q}\}_{n=j+a}^{j+b(r-1)+a}$'s $(1 \leq q \leq Q)$ 
are isomorphic to 
the constant sequence $\{\cO_{X}\}_{n=j+a}^{j+b(r-1)+a}$. 
Then, we can apply Proposition~\ref{prop:b} to 
the sequence $\{(F^*)^nE_X\}_{n=j+a}^{j+b(r-1)+a}$. 
So $(F^*)^{j+b(r-1)+a}E_X = \cO_X^r$ (hence 
$(F^*)^{c-1}E_X = \cO_X^r$)   in $\Coh(X)$.

Therefore,  $((F^*)^{c-1}E)_X$ has the form $(\cO_X^r, \nabla)$ when 
regarded as an object in $\MIC(X)^{\qn}$ via the equivalence \eqref{evpn}. 
Then, by \eqref{eqn:can}, one has
$$((F^*)^cE)_X = F^*(\cO_X^r, \nabla) = (\cO_X^r, d)$$ 
in $\MIC(X)^{\qn}$. So $((F^*)^cE)_X \in \Crys(X/k)$ is constant.

\end{proof}

Proposition~\ref{prop:c} deals with the value of a lattice of an isocrystal  in $\Crys(X/k)$. To go up to  $\Crys(X/W_n)$, 
 we consider the deformation theory of crystals. 

Let $X$ be a smooth projective  variety over $k$ and fix 
$n,r \in \N$. Let us denote the restriction functor 
$\Crys(X/W_{n+1}) \lra \Crys(X/W_n)$ by $G \mapsto \ol{G}$. 
Let $\cD$ be the set of pairs $(G, \varphi)$ consisting of 
$G \in \Crys(X/W_{n+1})$ and an isomorphism 
$\varphi: \cO_{X/W_n}^r  \os{\simeq}{\lra} \ol{G}$ in $\Crys(X/W_n)$. 
Then $\cD$ is a pointed set, whose distinguished element is 
$(\cO_{X/W_{n+1}}, \id)$. The pullback 
$(G, \varphi) \mapsto (F^*G, F^*\varphi)$ by $F^*$ 
defines a morphism of pointed sets 
$F^*: \cD \lra \cD$. We denote by $H^n_{\crys}(X/k)$ the 
crystalline cohomology of $X$ over $k$, which is the same as  the de Rham cohomology $H^n(X, \Omega^\bullet_{X/k})$. 
\begin{prop}\label{prop:e}
Let the notations be as above. Then there is an isomorphism of pointed sets 
$$ e: \cD \os{\simeq}{\lra} H^1_{\crys}(X/k)^{r^2}. $$
Moreover, the following diagram is commutative$:$ 
\begin{equation}\label{eq:e}
\begin{CD}
\cD @>e>> H^1_{\crys}(X/k)^{r^2} \\ 
@V{F^*}VV @V{F^*}VV \\ 
\cD @>e>> H^1_{\crys}(X/k)^{r^2}. 
\end{CD}
\end{equation}
\end{prop}

\begin{proof}
For $\ell = n, n+1$, let $D_\ell$ be the PD-envelope of 
the closed immersion $X \os{\iota}{\hra} \mathbb{P}^N_k 
\hra \mathbb{P}^N_{W_\ell}$. Using the equivalence 
$\Crys(X/W_\ell) \cong \MIC(D_\ell)^{\qn}$ of \eqref{evdn}, we consider 
the pointed set $\cD$ in terms of objects in $\MIC(D_\ell)^{\qn}$. Namely, 
we denote the restriction $\MIC(D_{n+1}) \lra \MIC(D_n)^{\qn}$ by 
$(G,\nabla) \mapsto (\ol{G}, \ol{\nabla})$ and we regard 
$\cD$ as the set of pairs $((G,\nabla), \varphi)$ consisting of 
$(G, \nabla) \in \MIC(D_{n+1})^{\qn}$ and an isomorphism 
$\varphi: (\cO_{D_n}, d)  \os{\simeq}{\lra} 
(\ol{G}, \ol{\nabla})$ in $\MIC(D_n)^{\qn}$. 

Assume given an object $G := ((G,\nabla), \varphi)$ in $\cD$. 
Take an affine open covering $\cU = \{U_{\alpha}\}_{\alpha}$ of $D_{n+1}$ and 
for each $\alpha$, take 
an isomorphism $\psi_{\alpha}: \cO_{U_{\alpha}}^r 
\os{\simeq}{\lra} G|_{U_{\alpha}}$ 
which lifts $\varphi|_{D_n \times_{D_{n+1}} U_{\alpha}}$. 
On each $U_{\alpha}$, the connection ${\psi_{\alpha}}^*(\nabla)$ 
is written as 
$d + p^nA_{\alpha}$, where $A_{\alpha} \in 
M_r(\Gamma(U_{\alpha},\Omega^1_{D_1}))$. 
On each $U_{\alpha\beta} := U_{\alpha} \cap U_{\beta}$,  the
gluing $(\psi_{\alpha}|_{U_{\alpha\beta}})^{-1} \circ 
(\psi_{\beta}|_{U_{\alpha\beta}})$ 
is given by $1+p^nB_{\alpha\beta}$, where $B_{\alpha\beta} 
\in M_r(\Gamma(U_{\alpha\beta},\cO_{D_1}))$. 
Then, $dA_{\alpha} = 0$ 
by the integrability of the connection ${\psi_{\alpha}}^*(\nabla)$ and 
$B_{\beta\gamma} - B_{\alpha\gamma} + B_{\alpha\beta} = 0$ 
by the cocycle condition for the maps 
$(\psi_{\alpha}|_{U_{\alpha\beta}})^{-1} \circ 
(\psi_{\beta}|_{U_{\alpha\beta}})$. 
Also, by the compatibility of the connection with the gluing, we have 
the equality 
$$ 
(1+p^nB_{\alpha\beta})^{-1}d(1+p^nB_{\alpha\beta}) + 
(1+p^nB_{\alpha\beta})^{-1}p^nA_{\alpha}(1+p^nB_{\alpha\beta}) 
= p^nA_{\beta}. $$
We see from this the equality $A_{\beta} - A_{\alpha} = dB_{\alpha\beta}$. 
So $(\{A_{\alpha}\}, \{B_{\alpha\beta}\})$ defines 
a $1$-cocycle of 
${\rm Tot}\,\Gamma(\cU,\Omega_{D_1}^{\bullet})^{r^2}$. 
We define $e(G)$ to be the class of this $1$-cocycle in 
the cohomology 
$H^1({\rm Tot}\,\Gamma(\cU,\Omega_{D_1}^{\bullet})^{r^2}) 
= H^1_{\crys}(X/k)^{r^2}$. 

In order  to show that this is well-defined, we need to check that 
$e(G)$ is independent of the choice of the 
affine open covering $\cU = \{U_{\alpha}\}_{\alpha}$ and 
the isomorphisms $\{\psi_{\alpha}\}_{\alpha}$. 
If we choose another set of isomorphisms 
$\{\psi'_{\alpha}\}_{\alpha}$, we have another set of 
matrices $(\{A'_{\alpha}\}, \{B'_{\alpha\beta}\})$. 
Then, 
on each $U_{\alpha}$, 
the map $\psi_{\alpha}^{-1} \circ \psi'_{\alpha}$ 
is given by $1+p^nC_{\alpha}$, where $C_{\alpha} 
\in M_r(\Gamma(U_{\alpha},\cO_{D_1}))$, and 
we see by direct calculation the equalities 
$dC_{\alpha} = A'_{\alpha} - A_{\alpha}$, 
$C_{\beta} - C_{\alpha} = B'_{\alpha\beta} - B_{\alpha\beta}$. 
So the class $e(G)$ does not depend on the choice of 
the isomorphisms $\{\psi_{\alpha}\}_{\alpha}$. 
One can prove the independence of the choice of affine open covering 
$\cU  = \{U_{\alpha}\}_{\alpha}$ by taking a refinement. 
So we obtain the map $e: \cD \lra  H^1_{\crys}(X/k)^{r^2}$, and 
it is easily seen that this is a map of pointed sets. 
One can prove the bijectivity of $e$ by considering the 
above argument in reverse direction. 

Finally, we prove the commuativity of the diagram \eqref{eq:e}. 
Let $F_{\mathbb{P}}: \mathbb{P}^N_{W_{n+1}} \lra \mathbb{P}^N_{W_{n+1}}$ 
be the 
$\sigma_{W_{n+1}}^*$-linear map which sends the coordinates to their 
$p$-th powers. Then, there exists a unique PD-morphism 
$F_{D_{n+1}}: D_{n+1} \lra D_{n+1}$ which makes the following diagram 
commutative: 
\begin{equation*}
\begin{CD}
X @>>> D_{n+1} @>>> \mathbb{P}^N_{W_{n+1}} \\ 
@V{F}VV @V{F_{D_{n+1}}}VV @V{F_{\mathbb{P}}}VV \\ 
X @>>> D_{n+1} @>>> \mathbb{P}^N_{W_{n+1}}. 
\end{CD}
\end{equation*}
Because $F_{D_{n+1}}$ mod $p$ is equal to the Frobenius map 
$F_{D_1}$ for $D_1$, we see 
(from the above expression of cocycle) that 
the class $e(G) = [(\{A_{\alpha}\}, \{B_{\alpha\beta}\})]$ 
is sent by $F^*$ (on cohomology) to 
$[(\{F_{D_1}^*A_{\alpha}\}, \{F_{D_1}^*B_{\alpha\beta}\})] 
= e(F^*(G))$. 
From this, we see the desired commutativity. 
\end{proof}

\begin{prop}\label{prop:d}
Let $X$ be a connected smooth projective  variety over $k$ 
with  trivial \'etale fundamental group. 
Let $\cE \in \Conv(X/K)$. 
Let $E$ be a lattice of $\cE$ such that 
the restriction $E_X \in \Crys(X/k)$ 
is constant. Then there exists a positive integer $d = d(X)$ 
such that, for any $n \in \N$,  the restriction
$((F^*)^{d(n-1)}E)_n$ of 
$(F^*)^{d(n-1)}E \in \Crys(X/W)$ to $\Crys(X/W_n)$ is 
constant.
\end{prop}

\begin{proof}
We have the decomposition 
$H^1_{\crys}(X/k) = H^1_{\crys}(X/k)_{\rm nilp} \oplus 
H^1_{\crys}(X/k)_{\rm ss}$ of $H^1_{\crys}(X/k)$ as in 
the proof of Proposition~\ref{prop:b}, where $H^1(X, \cO_X)$ is replaced by 
$H^1_{\crys}(X/k)$. As $F^*$ is $0$ on the image of $H^0(X, \Omega^1_{X/k})$ in 
$H^1_{\crys}(X/k)$, one has 
$ H^1_{\crys}(X/k)_{\rm ss}  \subset H^1(X, \cO_X)_{\rm ss}=0  $ by \eqref{eqn:O},
and there exists some $d \in \N$ 
such that $(F^*)^d$ acts by $0$ on 
$H^1_{\crys}(X/k)_{\rm nilp}$, 
since $H^1_{\crys}(X/k)_{\rm nilp}$ is finite-dimensional. 
So $(F^*)^d$ acts by $0$ on $H^1_{\crys}(X/k)$. 
We prove the proposition for this choice of $d$, by induction on $n$. 

Assume that 
$((F^*)^{c+d(n-1)}E)_n$ is 
constant. 
Then $((F^*)^{c+d(n-1)}E)_{n+1}$ defines the class 
$e(((F^*)^{c+d(n-1)}E)_{n+1})$ in $H^1_{\crys}(X/k)^{r^2}$ by 
Proposition \ref{prop:e}. Then, by definition of $d$, we have 
$0 = (F^*)^{d}e(((F^*)^{c+d(n-1)}E)_{n+1}) = 
e(((F^*)^{c+dn}E)_{n+1})$, and so $((F^*)^{c+dn}E)_{n+1}$
is constant again by Proposition~\ref{prop:e}.  This finishes the proof. 
\end{proof}

Combining Propositions \ref{prop:c} and \ref{prop:d}, we obtain the 
following: 

\begin{cor}\label{cor:d}
Let $X$ be a connected smooth projective variety over $k$ 
with trivial \'etale fundamental group.
 Let $r$ be a positive integer. 
Let $\cE \in \Conv(X/K)$  of rank $r$ 
and let $E$ be a lattice of $\cE$ such that 
$E_X \in \Coh(X)$ is strongly $\mu$-semistable.
Let $c = c(X,r)$ and $d = d(X)$ be as in Proposition~\ref{prop:c} 
and Proposition~\ref{prop:d}.  
Then, for any $n \in \N$, the restriction
$((F^*)^{c+d(n-1)}E)_n$ 
  of 
$(F^*)^{c+d(n-1)}E \in \Crys(X/W)$ to $\Crys(X/W_n)$ is 
constant.
\end{cor}

Now we can finish the proof of Theorem \ref{thm1}: 

\begin{proof}[Proof of Theorem \ref{thm1}]
Let $r$ be the rank of $\cE$ and let 
$c = c(X,r)$, $d = d(X)$ be 
as in Proposition~\ref{prop:c}, Proposition~\ref{prop:d} for 
$X$ and $r$. Apply Corollary~\ref{cor:d} to 
$\cE^{(c+d(n-1))}$ and its lattice 
$E^{(c+d(n-1))}$ for each $n \geq 1$. 
Then we see that 
the restriction of 
$G^{(n)} := (F^*)^{c+d(n-1)}E^{(c+d(n-1))}$ to 
$\Crys(X/W_n)$, which we denote by $G^{(n)}_n$, 
is constant. Note that, for any $n \geq 1$, this is a lattice of 
$(F^*)^{c+d(n-1)}\cE^{(c+d(n-1))} = \cE$. 

We put $E := G^{(1)}$ so that it is a lattice of $\cE$ with 
$E_1 \in \Crys(X/k)$ constant. 
We may assume by replacing $G^{(n)}$ by $p^{m_n}G^{(n)}$ 
for suitable $m_n \in \Z$ that $G^{(n)} \subseteq E,  \ 
G^{(n)} \not\subseteq pE$. 
We further fix a natural number $m\ge 1$.
Then, for any $n\ge m$,  the image of 
the composite map
\begin{align} 
H^0_{\crys}(X/W_n, G^{(n)}_n) & \lra 
H^0_{\crys}(X/W_n, E_n) \label{eq:thm1-1} \\ 
& \lra H^0_{\crys}(X/W_n, E_n/p^mE_n) 
= H^0_{\crys}(X/W_m, E_m) \nonumber 
\end{align}
is not zero: Otherwise, as
$G^{(n)}_n$ is constant, it would be contained in $p^mE_n\subset pE_n$. 
Hence $G^{(n)}$ is contained in $pE$, 
which is a 
contradiction. So, for $n\ge m$,  the map 
$$ H^0_{\crys}(X/W_n, E_n) 
\lra H^0_{\crys}(X/W_n, E_n/p^mE_n) 
= H^0_{\crys}(X/W_m, E_m) $$
is non-zero. Hence 
$$\{ {\rm Im}
\big(H^0_{\crys}(X/W_n,E_n) \lra H^0_{\crys}(X/W_m,E_m)\big)\}_{n \geq 1}$$ 
is a decreasing family of non-zero 
$W_m$-submodules  of the finite type $W_m$-module
$H^0_{\crys}(X/W_m,E_m)$. 
So, $W_m$ being an Artinian ring,  the family is stationary, thus non-zero,  and 
$$0 \neq \bigcap_{m\le n \in \N}{\rm Im}
\big(H^0_{\crys}(X/W_n,E_n) \lra H^0_{\crys}(X/W_m,E_m)\big).$$
 Thus the system
$\{H^0_{\crys}(X/W_n,E_n)\}_n$ satisfies the Mittag-Leffler condition and 
$$ 0\neq H^0_{\crys}(X/W,E) = \varprojlim_n H^0_{\crys}(X/W_n,E_n).$$

As $E$ is $p$-torsion free,  $H^0_{\crys}(X/W,E)$ is a free module of rank $s$ over $W$, for some $1\le s \le r$. 
If $s<r$, then the quotient $Q:=E/\big(H^0_{\crys}(X/W,E)\otimes_W \cO_{X/W}\big) \in \Crys(X/W)$ is nonzero. The $p$-torsion of $Q$
is identified with the kernel of the homomorphism 
$$(H^0_{\crys}(X/W,E)/p) \otimes_k \cO_{X/k} = 
\cO_{X/k}^s \to E_1 = H^0_{\crys}(X/k, E_1) \otimes_k \cO_{X/k}=\cO_{X/k}^r $$
in $\Crys(X/k)$, which is zero. Thus $Q \in \Crys(X/W)$ is $p$-torsion free. 
By multiplying the composite map $G^{(n)}\hookrightarrow E \twoheadrightarrow Q$ 
with a suitable $p$-power, we obtain a map $G^{(n)} \to Q$ whose image 
is not contained in $pQ$. Then the diagram \eqref{eq:thm1-1} with 
$E_n$ replaced by $Q_n$ shows that $H^0_{\crys}(X/W,Q)\neq 0$.
On the other hand, one has the exact sequence
$
0\to H^0_{\crys}(X/W,E)\otimes_W \cO_{X/W}\xrightarrow{\iota} E\xrightarrow{q} Q\to 0
$
 in $\Crys(X/W)$. By definition, $H^0_{\crys}(\iota)$ is an isomorphism and by Proposition~\ref{rem:eo} (2), $H^0_{\crys}(q)$ is surjective. Thus  $H^0_{\crys}(X/W,Q)=0$, a contradiction. Thus $s=r$ and $E$ is constant in $\Crys(X/W)$, thus $\cE$ is constant in $\Crys(X/W)_{\Q}$.  This finishes the proof. 

\end{proof}

We give a proof of Corollary \ref{cor}. 

\begin{proof}[Proof of Corollary \ref{cor}]
We check that any infinitesimal isocrystal $\cE = \Q \otimes E 
\in {\rm Inf}(X/W)_{\Q}$ satisfies the assumption of Theorem \ref{thm1}. 
By Proposition \ref{ber} below, the functor 
$F^*: {\rm Inf}(X/W) \lra {\rm Inf}(X/W)$ is an equivalence. 
Thus the $F^n$-division $\cE^{(n)}$ of $\cE$ has the form 
$\Q \otimes E^{(n)}$ for some $E^{(n)} \in {\rm Inf}(X/W)$. 
Then the value $E^{(n)}_X$ of $E^{(n)}$ at $X$ has the structure 
of an object in ${\rm Inf}(X/k)$, which is constant by 
the affirmation \cite{esnaultmehta} of Gieseker's conjecture. 
So $E^{(n)}_X$ is isomorphic to $\cO_X^r$ for some $r$ and hence 
strongly $\mu$-semistable. 
\end{proof}

\begin{prop}\label{ber}
For a smooth variety $X$ over $k$, the functor 
$$F^*: {\rm Inf}(X/W) \lra {\rm Inf}(X/W)$$ is an equivalence. 
\end{prop}

\begin{proof}
Because the category ${\rm Inf}(X/W)$ satisfies the Zariski 
descent property, we may work locally. 
So we may assume that $X$ lifts to a $p$-adic smooth formal 
scheme $X_W$ over $W$ on which there exists a lift 
$F_W:X_W \lra X_W$ of Frobenius morphism on $X$. 
Then we have the equivalence \eqref{1.8inf} 
in which the functor $F^*$ on the left hand side is compatible 
with the pull-back $F^*_W$ by $F_W$ on the right hand side. 
Thus it suffices to see that $F_W^*$ is an equivalence, which 
is proven in \cite[Thm.~2.1]{berthelotdivided}. 
\end{proof}

\section{Proof of Theorem \ref{thm2}}  \label{sec:proofthm2}

In this section, we prove Theorem \ref{thm2}. 
The following proposition, which is a crystalline version of 
Langton's theorem \cite{langton}, is a key step for the proof: 

\begin{prop}\label{prop:langton}
Let $X$ be a smooth projective  variety over $k$ and 
let $\cE \in \Crys(X/W)_{\Q}$ be irreducible.
Then there exists   $E \in  \Crys (X/W)$  with 
$\cE = \Q \otimes E$ such that $E_X \in \Crys(X/k) ={\MIC}(X)^{\qn}$ 
is $\mu$-semistable. 
\end{prop}

\begin{proof}
We follow the proof of Langer \cite[Thm.~5.1]{langerlie} and Huybrechts-Lehn's book 
\cite[2.B]{hl}. Let us consider the following two claims: 
\begin{itemize}
 \itemsep0em 
\item[(A)]  There exists  $E \in \Crys(X/W)$   with 
$\cE = \Q \otimes E$ such that $E_X \in \Coh(X)$ is torsion free.  
\item[(B)]  There exists 
$E \in \Crys(X/W)$   with 
$\cE = \Q \otimes E$ such that $E_X \in \Crys(X/k)={\MIC}(X)^{\qn}$
is $\mu$-semistable.
\end{itemize}
To prove the proposition, we first prove the claim (A) and then 
prove the claim (B). However, since the proof of (A) and that of (B) 
are parallel, we will describe them simultaneously in the following. \par 
First take a $p$-torsion free crystal $E \in \Crys(X/W)$ with 
$\cE = \Q \otimes E$ in the case (A), and take 
a $p$-torsion free crystal $E \in \Crys(X/W)$  with 
$\cE = \Q \otimes E$ 
and $E_X$ torsion free in the case (B). (This is possible because, when 
we prove (B), we can assume the claim (A).) 
Put $E^0 := E$. If $E^0$ does not 
satisfy the conclusion of the claim, let 
$B^0$ be the maximal torsion $\cO_X$-submodule of $E^0_X$ in the case (A) and 
the maximal destabilizing subobject of $E^0_X$ in the category $\Crys(X/k)$ in 
the case (B). In the case (A), one can check 
(by looking at $E^0_X$ as an object $(E^0_X,\nabla)$ in $\MIC(X)$ and noting 
the fact that $fe = 0$ \, $(e \in E^0_X, f \in \cO_X)$ 
implies $f^2\nabla(e) = 0$) that 
$B^0$ is actually an object in $\Crys(X/k)$. 
Let $E^1$ be the kernel of $E^0 \lra E^0_X \lra E^0_X/B^0$. 
If $E^1$ satisfies the conclusion of the claim, we are done. 
Otherwise, let 
$B^1$ be the maximal torsion $\cO_X$-submodule 
(actually an object in $\Crys(X/k)$) of $E^1_X$ in the case (A) and 
the maximal destabilizing subobject of $E^1_X$ 
in the category $\Crys(X/k)$ in the case (B). 
If the claim is not true, we obtain a sequence 
$$ E = E^0 \supset E^1 \supset E^2 \supset \cdots. $$ 

Let $G^n := E^n_X/B^n=E^n/E^{n+1}$.  Note that 
in the case (A), the rank of $G^n$ is the same as the rank of $E^n_X$, which is the same as the rank of $\cE$. In addition, as $B^n\subset E^n_X$ is the maximal torsion submodule, $G^n$ is torsion free in $\Coh(X)$. In the case (B),
$G^n$ is nonzero  by definition of $B^n$, and torsion free by the maximality of $B^n$. 
By definition, one has exact sequences $0\to E^{n+1}\to E^{n}\to G^n\to 0$
and $0\to pE^n/pE^{n+1} \to E_1^{n+1} \to E^{n+1}/pE^n \to 0, $
both in $\Crys(X/W)$. As  $pE^n/pE^{n+1}\cong G^n$ and $E^{n+1}/pE^n=B^n$ in $\Crys(X/W)$, this yields the exact sequences 
\begin{equation}\label{eq:l1}
0 \lra B^n \lra E^n_X \lra G^n \lra 0, \,\,\,\, 
0 \lra G^n \lra E^{n+1}_X \lra B^n \lra 0 
\end{equation}
in $\Crys(X/k)$. 
From these, we see that the slope $\mu(E^n_X)$ of $E^n_X$ is constant and so 
equal to $\mu(E_X)$ in the case (B). 

Let $C^n$ be the kernel of the composite 
$B^{n+1} \ra E^{n+1}_X \ra B^n$. It is 
nothing but $B^{n+1} \cap G^n$, and this is zero in the case (A) 
because $B^{n+1}$ is torsion while $G^n$ is torsion free. 
In the case (B), 
if $C^n = 0$, $\mu(B^{n+1}) \leq \mu(B^n)$ due to the maximality of 
$B^n$. 
If $C^n \not= 0$, 
$\mu(C^n) \leq \mu_{\max}(G^n) < \mu(B^n)$ because $C^n$ is a subobject 
of $G^n$ and $B^n$ is the maximal destabilizing subobject. 
So, if $\mu(B^{n+1}) \leq \mu(C^n)$, we obtain the inequality 
$\mu(B^{n+1}) < \mu(B^n)$. On the other hand, if 
$\mu(B^{n+1}) > \mu(C^n)$, we have 
$\mu(B^{n+1}) < \mu(B^{n+1}/C^n) \leq \mu(B^n)$ because $B^n$ is 
$\mu$-semistable as  a crystal. Hence $\mu(B^{n+1}) < \mu(B^n)$ when 
$C^n \not= 0$. In conclusion, 
$\mu(B^n) \,(n \in \N)$ is non-increasing, and strictly decreasing when 
$C^n \not= 0$. But the latter case can happen only finitely many times, 
because $\mu(B^{n})$ should be contained in $\frac{1}{r!}\Z$ (where 
$r$ is the rank of $E$) and $> \mu(E_1)$. Therefore, 
$C^n = 0$ for $n \gg 0$ in the case (B). 

So we may assume that $C^n = 0$, namely, $B^{n+1} \cap G^n = 0$. 
This implies that we have the inclusions 
\begin{equation}\label{eq:l2}
\cdots \supseteq B^n \supseteq B^{n+1} \supseteq \cdots, \qquad 
\cdots \subseteq G^n \subseteq G^{n+1} \subseteq \cdots. 
\end{equation}
We may assume also that the rank of $G^n$ is constant and that 
$\mu(B^n) \,(n \in \N), \mu(G^n) \, \allowbreak (n \in \N)$ are constant in 
the case (B).  Note also that $G^n = G^{n+1}$ if and only if $B^{n} = B^{n+1}$. 

Next we prove that $G^n$ is constant for $n \gg 0$. 
In the case (A), the support of $B^n$ is non-increasing and so it is 
constant for $n \gg 0$. So, for $n \gg 0$, 
$B^n = B^{n+1}$ outside some codimension $2$ closed subscheme of $X$.
 Indeed,  if the support of the $B_n$ for $n$ large is in codimension $\ge 2$, there is nothing to prove, else it is a divisor, and $B_n$ on each generic point of the divisor is eventually constant. So $G^n = G^{n+1}$ outside a codimension $2$ closed subscheme. 
Hence the double dual of $G^n$ is constant and, as $G^n$ is torsion free, contains all the $G^n$. 
So the right tower   in \eqref{eq:l2} is stationary and then 
$G^n$ is constant for $n \gg 0$. 
In the case (B), the constancy of the rank and the slope and the torsion freeness of 
$G^n$ imply the equality $G^n = G^{n+1}.$ 

So we may assume that $B^n, G^n$ are constant. So we write it by 
$B, G$, respectively. Then the exact sequences \eqref{eq:l1} 
split, and so $E^n_X = B \oplus G$. Now define $Q_n := E/E^n$. 
Then $Q$ has a natural filtration whose graded quotients are 
$E^i/E^{i+1} \cong G$. This implies that $Q_n$ is nonzero and when regarded as an object in 
$\Crys(X/W_n)$, it is 
flat over $W_n$. So 
$Q = (Q_n)_n \in \Crys(X/W)$ is a nonzero $p$-torsion free crystal. 
Also, we have the canonical surjection $E \lra Q$, hence the surjection 
$\cE \lra \Q \otimes Q$. In the case (A), if it is not an isomorphism, 
this contradicts the irreducibility of $\cE$. If it is an isomorphism, 
$Q$ gives the lattice such that $Q_X = Q_1 = G$ is torsion free. 
In the case (B), since 
$B$ is non-zero and torsion free, $\cE \lra \Q \otimes Q$ 
is not an isomorphism, and this contradicts the irreducibility of $\cE$. 
This  finishes the proof.
\end{proof}

\begin{prop}\label{prop:mr}
Let $X$ be a  smooth projective variety over $k$ and 
let $G\in \Crys(X/k)$ be  of rank $r$
and $\mu$-semistable. 
Assume moreover  one of the following conditions:
\begin{itemize}
\itemsep0em 
\item[(1)] $r=1$. 
\item[(2)]  $\mu_{\max}(\Omega^1_X) < 2$, $r=2$ and $\mu(G) = 0$. 
\item[(3)]  $\mu_{\max}(\Omega^1_X) < 1$, $r=3$ and $\mu(G) = 0$. 
\item[(4)]  $\mu_{\max}(\Omega^1_X) < \dfrac{1}{N(r)}$, where 
$N(r) := \displaystyle\max_{a, b \geq 1, a+b \leq r} {\rm lcm}(a,b)$. 
\end{itemize}
Then $G$ is strongly $\mu$-semistable in $\Coh(X)$.
\end{prop}


\begin{proof}
In the case (1), the only point of the assertion is that the Frobenius pull backs of $G$ remain torsion free, which is trivial because torsion freeness of 
$\cO_X$-modules is preserved by $F^*$ 
as $X$ is smooth so $F^*$ is faithfully flat. So we will prove 
the proposition in the cases (2), (3) or (4). 
The proof is a variant of that in \cite[Thm.~2.1]{mr}. 

First we prove the claim that any $\mu$-semistable sheaf $H$ of rank $r$ is 
strongly $\mu$-semistable in $\Coh(X)$ under one of the following conditions: 
\begin{itemize}
\itemsep0em 
\item[(a)]  $\mu_{\max}(\Omega^1_X) < 2$, $r=2$ and $\mu(H) = 0$. 
\item[(b)]  $\mu_{\max}(\Omega^1_X) < 1$, $r=3$ and $\mu(H) = 0$. 
\item[(c)]  $\mu_{\max}(\Omega^1_X) < \dfrac{1}{N(r)}$. 
In this case, we prove it by induction on $r$.
\end{itemize}

For this, it suffices to prove that $F^*H$ is $\mu$-semistable in $\Coh(X)$. 
Assume the contrary and let $H' \subset F^*H$ be the 
maximal destablizing subsheaf of $H$. Let 
$H'' := H/H'$. Then the connection
$\nabla_{\rm can}: F^*H \lra F^*H \otimes \Omega^1_X$ in \eqref{cartierd}
induces a linear map 
$\ol{\nabla}_{\rm can}: H' \lra H'' \otimes \Omega^1_X$. 
If we prove $\ol{\nabla}_{\rm can} = 0$, 
$(H', \nabla_{\rm can}|_{H'})$ defines a submodule with 
integrable connection of $(F^*H, \nabla_{\rm can})$ and so 
there exists a $\cO_X$-submodule $H'_0$ of $H$ with 
$H' = F^*H'_0$. Then we have 
$p \mu(H'_0) = \mu(H') > \mu(F^*H) = p\mu(H)$ and this 
contradicts the $\mu$-semistability of $H$. So 
it suffices to prove the equality $\ol{\nabla}_{\rm can} = 0$. 
To prove it, we may replace 
$H''$ by its graded quotients with respect to 
Harder-Narasimhan filtration. So we may assume that 
$H', H''$ are $\mu$-semistable and $\mu(H') > \mu(H'')$. 
Also, it suffices to prove that the map 
$$ f: T_X \lra {\cal H}om(H',H'') $$
(where $T_X$ denotes the tangent sheaf on $X$)
induced by $\ol{\nabla}_{\rm can}$ is equal to zero. 
Since $T_X$ is locally free and ${\cal H}om(H',H'')$ is torsion free as $H''$ is, 
it suffices to prove $f=0$ outside some codimension $2$ closed subscheme 
of $X$. Until the end of the proof of the claim, we will consider 
sheaves and morphisms of sheaves up to some codimension $2$ subscheme 
in $X$. Then ${\cal H}om(H',H'') \cong {H'}^{\vee} \otimes H''$. 
When (a) or (b) is satisfied, at least one of $H', H''$ is 
of rank $1$. So ${H'}^{\vee} \otimes H''$ is $\mu$-semistable 
of slope $- \mu(H') + \mu(H'')$, which is $\leq -2$ in 
the case (a) and $\leq -1$ in the case (b) (we use the assumption 
$\mu(H) = 0$ here). 
Hence $- \mu(H') + \mu(H'') < -\mu_{\max}(\Omega^1_X) = \mu_{\min}(T_X)$. 
So we see that $f=0$.  
%

When (c) is satisfied, $H', H''$ are strongly $\mu$-semistable 
by induction hypothesis. Then, 
by \cite[Thm.~3.23]{rr} 
(see also \cite[Cor.~A.3.1]{langermixed}), 
${H'}^{\vee} \otimes H''$ 
is $\mu$-semistable of slope $- \mu(H') + \mu(H'') \leq 
\dfrac{-1}{{\rm lcm}({\rm rank} H', {\rm rank} H'')} \leq -\dfrac{1}{N(r)} 
< -\mu_{\max}(\Omega^1_X) = \mu_{\min}(T_X)$. 
So we see that $f=0$ also in this case. 

Now we prove the proposition. 
In the proof, we regard $G$ as an object in $\MIC(X)^{\qn}$ and so 
we denote it by $(G,\nabla)$. By the argument above, 
it suffices to prove that 
$G$ is $\mu$-semistable as sheaf. 
Assume the contrary and let $H' \subset G$ be the 
maximal destablizing subsheaf of $G$. Let 
$H'' := G/H'$. Then the connection 
$\nabla: G \lra G \otimes \Omega^1_X$ induces a linear map 
$\ol{\nabla}: H \lra H' \otimes \Omega^1_X$. 
It suffices to prove that $\ol{\nabla} = 0$: Indeed, if this is 
the case, $(H, \nabla|_H)$ defines a destabilizing 
subobject of $(G, \nabla)$, which is a contradiction. 

We prove that $\ol{\nabla} = 0$ in a similar way to 
the proof of $\ol{\nabla}_{\rm can} = 0$ above. 
We replace 
$H''$ by its graded quotients with respect to 
Harder-Narasimhan filtration so that 
$H', H''$ are $\mu$-semistable and $\mu(H') > \mu(H'')$, and 
we prove that the map 
$$ f: T_X \lra {\cal H}om(H',H'') $$
induced by $\ol{\nabla}$ is equal to zero 
outside some codimension $2$ closed subscheme 
of $X$. Working again up to some codimension $2$ subscheme 
in $X$, we have ${\cal H}om(H',H'') \cong {H'}^{\vee} \otimes H''$. 
When (1) or (2) is satisfied, at least one of $H', H''$ is 
of rank $1$. So ${H'}^{\vee} \otimes H''$ is $\mu$-semistable 
of slope $- \mu(H') + \mu(H'')$, which is $\leq -2$ in 
the case (1) and $\leq -1$ in the case (2) (we use the assumption 
$\mu(G) = 0$ here). In the case (3), $H', H''$ are strongly $\mu$-semistable 
by the claim we proved above. Then, 
${H'}^{\vee} \otimes H''$ 
is $\mu$-semistable of slope $- \mu(H') + \mu(H'') \leq 
\dfrac{-1}{{\rm lcm}({\rm rank} H', {\rm rank} H'')} \leq -\dfrac{1}{N(r)} 
< -\mu_{\max}(\Omega^1_X) = \mu_{\min}(T_X)$. 
So we see that $f=0$ also in this case.

\end{proof}

Now we give a proof of Theorem \ref{thm2}: 

\begin{proof}[Proof of Theorem \ref{thm2}] 

First assume that $\cE$ is irreducible. In this case, 
any $F^n$-division $\cE^{(n)}$ of $\cE$ is also irreducible. 
Then, by Propositions \ref{prop:chern}, \ref{prop:langton} and \ref{prop:mr}, 
each $\cE^{(n)}$ admits a lattice $E^{(n)}$ such that 
$E^{(n)}_X$ is strongly $\mu$-semistable as $\cO_X$-module. 
So, by Theorem \ref{thm1}, we see that $\cE$ is constant. 

In the general case, $\cE$ has a filtration whose graded quotients are 
irreducible. So, by the previous case, $\cE$ can be written as 
an iterated extension of constant convergent isocrystals. 
Since we have $H^1_{\conv}(X/K, \cO_{X/K}) 
= \Q \otimes H^1_{\crys}(X/W, \cO_{X/W}) = 0,$ 
 where the second equality is proven in 
Proposition~\ref{rem:eo} (2), 
this finishes the proof. 
\end{proof}

We give another application of Proposition 
\ref{prop:langton}. It seems that  
the following question is frequently asked among 
experts: 
\begin{question}
Let $X$ be a smooth variety of finite type over $k$ and let 
$\cE \in \Conv(X/K)$. Does there exist 
a locally free $E\in \Crys(X/W)$  with 
$\cE = \Q \otimes E$? 
\end{question}

We give the following partial answer to this question, 
using Proposition~\ref{prop:langton}: 

\begin{thm}\label{thm:lf}
Let $X$ be a  smooth projective variety over $k$, let 
$\cE \in \Conv(X/K)$. Assume one of the 
following: 
\begin{itemize}
\itemsep0em 
\item[(1)] The rank of 
irreducible constituents of $\cE$ are $1$. 
\item[(2)] $\mu_{\max}(\Omega_X^1) < 2$ and the rank of 
irreducible constituents of $\cE$ are $\leq 2$. 
\item[(3)]  $\mu_{\max}(\Omega_X^1) < 1$ and the rank of 
irreducible constituents of $\cE$ are $\leq 3$. \\ 
\item[(4)]  $r \geq 4$, $\mu_{\max}(\Omega_X^1) < 
\dfrac{1}{N(r)}$ and the rank of 
irreducible constituents of $\cE$ are $\leq r$, where 
$N(r) := \displaystyle\max_{a, b \geq 1, a+b \leq r} {\rm lcm}(a,b)$. \\ 
\item[(5)]  $X$ lifts to a smooth scheme $\wt{X}$ over $W_2$ and the rank of 
irreducible constituents of $\cE$ are $\leq p$. 

Then there exists  $E \in \Crys(X/W)$   locally free with 
$\cE = \Q \otimes E$. 
\end{itemize}
\end{thm}

The case (1) reproves a weaker version of 
Proposition~\ref{prop:rk1lattice}. Also, 
when $\mu_{\max}(\Omega_X^1) \leq 0$, 
any convergent isocrystal $\cE$ on $X$ admits 
a locally free crystal $E$ on $(X/W)_{\crys}$ with 
$\cE = \Q \otimes E$. 

\begin{proof}
First we prove the theorem in the cases 
(1), (2), (3) or (4) by induction on the rank of $\cE$. 
When $\cE$ is irreducible, there exists a lattice $E$ of 
$\cE$ such that $E_X \in \Coh(X)$ is strongly $\mu$-semistable, by 
Propositions~\ref{prop:chern}, 
\ref{prop:langton} and \ref{prop:mr}. This, together with 
Proposition \ref{prop:chern}, Theorem~\ref{thm:langer} implies that 
$E_X$ is locally free. Hence $E$ is also locally free. When 
$\cE$ is not irreducible, 
we have an irreducible convergent subisocrystal $\cE' \subsetneq \cE$. 
Put $\cE'' := \cE/\cE'$. 
Then, by induction hypothesis, there exist locally free lattices 
$E', E''$ of $\cE', \cE''$, respectively. Then 
$H^1_{\conv}(X/K, {\cE''}^{\vee} \otimes \cE') = 
\Q \otimes H^1_{\crys}(X/W, {E''}^{\vee} \otimes E')$, and 
from this we see that there exists an extension $E$ 
of $E'$ by $E''$ in $\Crys(X/W)$ with $\cE \cong \Q \otimes E$. 
This $E$ is locally free by construction, and so 
the theorem is true for $\cE$. 

Next we prove the theorem in the case (5). 
By the argument in the previous paragraph, 
we may assume that $\cE$ is irreducible. Using 
Proposition~\ref{prop:langton}, we take a lattice $E$ of 
$\cE$ such that the restriction $(E_X, \nabla)$ of $E$ to 
$\Crys(X/k) = \MIC(X)^{\qn}$ is 
$\mu$-semistable. By assumption, ${\rm rank}\,E_X \leq p$. 
Hence $(E_X, \nabla)$ is contained in $\MIC_{p-1}(X)$. 
So, by \cite{ov}, 
there exists a Higgs module $(H, \theta) \in \HIG_{p-1}(X)$ such that 
$C^{-1}(H, \theta) = (E_X, \nabla)$, where 
$C^{-1}$ is the inverse Cartier transform. By Proposition~\ref{prop:chern}, $E_X$ has vanishing Chern classes. 
From this and the $\mu$-semistability of $(E_X, \nabla)$, we see 
that $(H, \theta)$ is $\mu$-semistable Higgs module with 
vanishing Chern classes, by \cite[Lem.~2, Cor.~1]{langerhiggs}. 
Then, by \cite[Thm.~11]{langerhiggs}, 
$H$ is locally free. Hence so is $E_X$, and then $E$ is locally free. 
\end{proof}

\end{document}